\documentclass[11pt]{amsart}

\usepackage{amsxtra,amssymb,amsmath,amscd,url,listings,stmaryrd}
\usepackage[alphabetic,initials]{amsrefs}
\usepackage{scalerel,stackengine}
\stackMath
\usepackage[utf8]{inputenc}
\usepackage{eucal}
\usepackage{fullpage}
\usepackage{scrtime}
\usepackage[colorlinks]{hyperref}
\usepackage{datetime}
\usepackage{graphicx}

\shortdate

\renewcommand{\leq}{\leqslant}
\renewcommand{\geq}{\geqslant}
\newcommand\widecheck[1]{%
\savestack{\tmpbox}{\stretchto{%
  \scaleto{%
    \scalerel*[\widthof{\ensuremath{#1}}]{\kern-.4pt\bigwedge\kern-.4pt}%
    {\rule[-\textheight/2]{1ex}{\textheight}}
  }{\textheight}%
}{0.5ex}}%
\stackon[2pt]{#1}{\scalebox{-1}{\tmpbox}}%
}

\numberwithin{equation}{section}

\def\stacksum#1#2{{\stackrel{{\scriptstyle #1}}
{{\scriptstyle #2}}}}

\newcommand{\FT}{\mathrm{FT}}
\newcommand{\Sym}{\mathrm{Sym}}
\newcommand{\sym}{\mathrm{sym}}

\newcommand{\Cc}{\mathbf{C}}

\newcommand{\Ct}{\mathbf{C}^\times}

\newcommand{\Aa}{\mathbf{A}}
\newcommand{\Gm}{\mathbb{G}_{m}}

\newcommand{\Zz}{\mathbf{Z}}
\newcommand{\Pp}{\mathbf{P}}
\newcommand{\Rr}{\mathbf{R}}

\newcommand{\Ql}{\mathbf{Q}_{\ell}}

\newcommand{\Fq}{{\mathbf{F}_q}}

\newcommand{\Fqt}{{\mathbf{F}^\times_q}}

\newcommand{\mcW}{\mathcal{W}}

\newcommand{\HYPK}{\mathcal{K}\ell}
\newcommand{\KL}{\mathcal{K}\ell}

\newcommand{\mods}[1]{\,(\mathrm{mod}\,{#1})}

\newcommand{\what}{\widehat}

\newcommand{\ra}{\rightarrow}

\DeclareMathOperator{\rank}{rank}

\DeclareMathOperator{\Kl}{\mathrm{Kl}}

\DeclareMathOperator{\sw}{Swan}

\newcommand{\eps}{\varepsilon}
\renewcommand{\rho}{\varrho}

\DeclareMathOperator{\SL}{SL}

\DeclareMathOperator{\GL}{GL}

\DeclareMathSymbol{\gena}{\mathord}{letters}{"3C}
\DeclareMathSymbol{\genb}{\mathord}{letters}{"3E}

\def\intc{\frac{1}{2i\pi}\mathop{\int}\limits}

\newcommand{\qoo}{q^{o(1)}}

\theoremstyle{plain}
\newtheorem{theorem}{Theorem}[section]
\newtheorem*{theorem*}{Theorem}
\newtheorem{lemma}[theorem]{Lemma}

\newtheorem{corollary}[theorem]{Corollary}

\newtheorem{proposition}[theorem]{Proposition}

\theoremstyle{remark}

\theoremstyle{definition}

\newtheorem{remark}[theorem]{Remark}

\newcommand{\mcL}{\mathcal{L}}

\newcommand{\mcC}{\mathcal{C}}

\newcommand{\rmL}{\mathrm{L}}
\newcommand{\mcF}{\mathcal{F}}
\newcommand{\mcZ}{\mathcal{Z}}

\newcommand{\mcK}{\mathcal{K}}

\newcommand{\mcJ}{\mathcal{J}}

\newcommand{\mcG}{\mathcal{G}}

\newcommand{\lf}{\lambda_f}

\newcommand{\vphi}{\varphi}

\renewcommand{\geq}{\geqslant}
\renewcommand{\leq}{\leqslant}
\renewcommand{\Re}{\mathfrak{Re}\,}

\newcommand{\ov}[1]{\overline{#1}}

\newcommand\sumsum{\mathop{\sum\sum}\limits}

\newcommand{\sumstar}{\sideset{}{^\star}\sum}
\newcommand{\sumsumstar}{\sideset{}{^\star}\sumsum}

\begin{document}
\title{Algebraic twists of $\GL_3\times\GL_2$ $L$-functions}

\author{Yongxiao Lin}
\address{EPFL/MATH/TAN, Station 8, CH-1015 Lausanne, Switzerland }
\email{yongxiao.lin@epfl.ch}

\author{Philippe Michel}
\address{EPFL/MATH/TAN, Station 8, CH-1015 Lausanne, Switzerland }
\email{philippe.michel@epfl.ch}

\author{Will Sawin}
\address{Department of Mathematics, Columbia University, 2990 Broadway, New York, NY 10027, USA }
\email{sawin@math.columbia.edu}


\begin{abstract} 
  We prove that the coefficients of a $\GL_3\times\GL_2$ Rankin--Selberg $L$-function 
  do not correlate with a wide class of trace functions of small conductor modulo
  primes, generalizing the corresponding result~\cite{FKM1}
  for~$\GL_2$ and \cite{KLMS} for $\GL_3$. This result is inspired  by a recent work of P. Sharma who discussed the case of a Dirichlet character of prime modulus. 
\end{abstract}

\thanks{Y. L. and Ph.\ M. were partially supported by a
  DFG-SNF lead agency program grant (grant
  200020L\_175755) and the SNF (grant 200021\_197045). W. S. served as a Clay Research Fellow while working on this paper. To appear in American Journal of Math. \today}

\maketitle

\tableofcontents

\section{Introduction}
We begin by describing the content of a recent preprint of P. Sharma \cite{sharma} which is the starting point of the present work.

Let $(\lambda(r,n))_{r,n}$ be the Hecke eigenvalues of a $\GL_3$ cusp form $\vphi$ and $(\lf(m))_m$ be Hecke eigenvalues of a $\GL_2$ cusp form $f$ (holomorphic or Maass); for simplicity we assume that both have level $1$ (i.e., are $\SL_3(\Zz)$ and $\SL_2(\Zz)$-invariant respectively). Their Rankin--Selberg $L$-function is the Dirichlet series
$$L(\vphi\times f,s)=\sum_{n,r\geq 1}\frac{\lambda(r,n)\lf(n)}{(nr^2)^s},\ \Re s>1.$$
This has an Euler product of degree $6$ admitting analytic continuation to $\Cc$ and a functional equation relating $L(\vphi\times f,s)$ to $L(\ov{\vphi}\times \ov{f},1-s).$

Given $q$ a prime and  $\chi:(\Zz/q\Zz)^\times\ra\Ct$ a non-trivial Dirichlet character, the twisted $L$-function is 
$$L(\vphi\times f\times\chi,s)=\sum_{n,r\geq 1}\frac{\lambda(r,n)\lf(n)\chi(nr^2)}{(nr^2)^s},\ \Re s>1.$$
This again has an Euler product of degree $6$ admitting analytic continuation to $\Cc$ and a functional equation relating $L(\vphi\times f\times\chi,s)$ to $L(\ov{\vphi}\times \ov{f}\times\ov \chi,1-s).$ 

The following bound, known as the {\em convexity bound} (in the $q$-aspect) is not too hard to establish: for $\Re s=1/2$, one has
$$L(\vphi\times f\times\chi,s)\ll_{f,\vphi,s} q^{3/2+o(1)}.$$
The subconvexity problem aims at improving the exponent $3/2$. In \cite{sharma}, P. Sharma provided a detailed description of a solution of this problem with the bound 
$$L(\vphi\times f\times\chi,s)\ll_{f,\vphi,s} q^{3/2-1/16+o(1)}.$$
Inspired by the previous works \cite{FKM1,KLMS}, we follow Sharma's strategy and generalise this bound by replacing $\chi$ by a generic {\em trace function} 
$$K:(\Zz/q\Zz)^\times\ra\Cc,$$
that is the (restriction to $\Fqt$ of the) Frobenius trace function associated to some geometrically irreducible middle extension sheaf $\mcF$  on $\Pp^1_{\Fq}$ pure of weight $0$ satisfying additional generic condition (we call such a sheaf ``good").

Indeed (see \S \ref{secdual}), it follows from the (Mellin) expansion of $K_{|\Fqt}$ into Dirichlet characters that the series
$$L(\vphi\times f\times K,s)=\sum_\stacksum{n,r\geq 1}{(nr^2,q)=1}\frac{\lambda(r,n)\lf(n)K(nr^2)}{(nr^2)^s},\ \Re s>1,$$
has analytic continuation to $\Cc$ and satisfies a functional equation relating $L(\vphi\times f\times K,s)$ to
$L(\ov{\vphi}\times \ov{f}\times \widecheck{K}^6,s)$ where $\widecheck{K}^6:\Fqt\ra\Cc$ is a suitable ``$\GL_6$" transform of $K$ (see \eqref{hatK6}). In most cases, $\widecheck{K}^6$ is essentially (the restriction to $\Fqt$ of) a trace function. From this one can deduce a ``convexity" bound
\begin{equation}\label{ConvexKtwisted}
L(\vphi\times f\times K,s)\ll_{f,\vphi,s} q^{3/2+o(1)},\ \Re s=1/2.	
\end{equation}
 Our goal is to improve the exponent $3/2$.

Using approximate functional equation techniques, one sees that solving this ``subconvexity problem"  for  $L(\vphi\times f\times K,s)$ is tantamount to bounding non-trivially sums of the shape
$$S^t_V(K,X):=\sum_{r,n}\lambda(r,n)\lf(n)K(nr^2)V(\frac{nr^2}{X})$$
for $V$ a (fixed) smooth function with compact support in $[1,2[$ and  $X\geq 1$ a positive parameter of size $X\approx q^3$ (the trivial bound being $|S^t_V(K,X)|\leq X^{1+o(1)}$). Indeed when $\vphi$, $f$ and $s$ are fixed, $q^3$ is approximately the square-root of the conductor of the degree $6$ $L$-functions $L(\vphi\times f\times \chi,s)$ for the primitive Dirichlet characters $\chi\mods q$.	

For this we introduce the following conditions on the sheaf $\mcF$ both of which are generic (i.e., hold for a ``typical" sheaf). Let $[\times\lambda] \colon \mathbb A^1 \to \mathbb A^1$ denote the map which multiplies the coordinate by $\lambda$.
\begin{itemize}
\item (MO) There is no any $ \lambda \in \Fqt$ such that the geometric monodromy group of $\mcF$ has some quotient which is equal, as a representation of the geometric fundamental group $\pi_1$ into an algebraic group, to the geometric monodromy group of the Kloosterman sheaf $[\times \lambda]^*\KL_2$ modulo $\pm 1$.

\item (SL) The local monodromy representation of $\mathcal F$ at $\infty$ has no summand with slope $1/2$.
\end{itemize}

We will also need to assume that the sheaf $\mcF$ is {\em Fourier}, {\em i.e.}, that it is not geometrically isomorphic to either the constant sheaf or any Artin--Schreier sheaf (whose trace functions are additive characters).

 We will say that $\mcF$ is {\em good} if $\mcF$ is Fourier and satisfies both of (MO) and (SL) and say that $\mcF$ is bad otherwise.

Our main result provides bounds for the sums $S^t_V(K,X)$ when $K$ is the trace function of a good sheaf, $X$ is not too far from $q^3$ and when the smooth function $V$ is allowed to vary and oscillate mildly as $q$ varies. 

\begin{theorem}\label{thmStVbound}  Let $q$ be a prime and $K$ be the trace function modulo $q$ associated with a good sheaf $\mcF$. Let $Z\geq 1$ be some parameter and $V\in\mcC^\infty_c(\Rr)$ be a smooth function compactly supported in the interval $[1,2[$ and satisfying for all $i\geq 0$
\begin{equation}\label{thmeqtestfct}
V^{(i)}(x)\ll_i Z^i.	
\end{equation}
Let $X\geq 1$ be such that $Z^4q^{11/4}<X<Z^4q^{\frac{7-\theta_3}{2}}$. We have 
	\begin{equation}\label{StVbound}
	S^t_V(K,X)\ll \qoo \left(Z{X^{3/4}q^{11/16}}+Z^{\frac{4(1-\theta_3)}{3-2\theta_3}}X^{\frac{2-\theta_3}{3-2\theta_3}}q^{\frac{11(1-\theta_3)}{4(3-2\theta_3)}}+
Xq^{-1/8}\right);	
	\end{equation}
	here the implicit constant depends on $\vphi,f$, $C(\mcF)$ and the implicit constant in \eqref{thmeqtestfct} and $\theta_3=5/14$ is the best known bound towards the Ramanujan--Petersson conjecture on $\rm GL_3$.
In particular for $X=q^3$ (the convexity range) one obtains 
  $$S^t_V(K,q^3)\ll Zq^{3-1/16+o(1)}.$$
	
\end{theorem}
\begin{remark}
(1) This bound is non-trivial (i.e., is $o(X)$) as long as
$$X\gg q^{3-1/4+\eta},\hbox{ for some }  \eta>0.$$
The assumption $X<Z^4q^{(7-\theta_3)/2}$ in the statement is non-essential and possibly can be removed. We make  this assumption only to simplify our treatment at one certain point; see \eqref{Lbound}.
\par
(2) For $k\geq 2$ an integer, the hyper-Kloosterman sheaf $\KL_k$ whose attached trace function is given by the $k-1$-dimensional hyper-Kloosterman sums
$$\Kl_k(n; q)=\frac{1}{q^{\frac{k-1}{2}}}\sumsum_\stacksum{x_1,\cdots,x_k\in\Fqt}{x_1.\cdots.x_k=n}e\big(\frac{x_1+\cdots+x_k}{q}\big)$$
is good unless $k=2$.
In that case neither (SL) nor (MO) holds. We will explain in \S \ref{secdual} how a duality principle which gives the analytic continuation of $L(\vphi\times f\times K,s)$ allows for partial results for such sums.

\par (3) As was pointed out to us by V. Blomer, when $\vphi=\Sym_2(g)$ is the symmetric square lift of a $\GL_2$-modular form $g$ and taking $f$ to be a suitable Eisenstein series, one can obtain --by a variation on \cite{BlomerAJM}, using the Petrow--Young variant of the Conrey--Iwaniec method \cite{CI,PY}-- the stronger subconvex bound
$$|L(\Sym_2(g)\times \chi,s)|^2\ll q^{3/2-1/4+o(1)},\ \Re s=1/2$$
for any Dirichlet character $\chi\mods q$. Since this approach uses positivity of central values, it is not yet clear whether this could be extended to general trace functions. 
\end{remark}

\subsection{Principle of the proof}
To illustrate the main idea of our approach, we provide a quick sketch of the proof, focusing on just the ``generic" case in various transformations. Therefore we will assume $X\asymp q^3$ and $r=1$, and we will suppress the smooth test functions from our notation. We denote the $\GL_3$ Hecke eigenvalues by $(\lambda(1,n))_n$ and the $\GL_2$ Hecke eigenvalues by $(\lambda(n))_n$. 

Let $L\geq 1$ be a parameter (a positive power of $q$) whose optimal size to be determined later, and let $\ell\in [L,2L[$ be prime numbers. We use the Hecke relation to write $\lambda(1,n\ell)\approx \lambda(1,n)\lambda(1,\ell)$ and then use the Kronecker symbol to separate oscillations of $\lambda(1,n\ell)$ and $\lambda(n)K(n)$. Our starting point is to follow \cite{sharma} and write
\begin{equation*}
\begin{split}
\mathcal{S}:=&\sum_{n\sim q^3}\lambda(1,n)\lambda(n)K(n)\\
\approx& \frac{1}{L} \sum_{\ell\sim L}\ov{\lambda(1,\ell)}\sum_{n\sim q^3}\lambda(1,n\ell)\lambda(n)K(n)\\
=&\frac{1}{L} \sum_{\ell\sim L}\ov{\lambda(1,\ell)}\sum_{n\sim q^3 \ell}\lambda(1,n)\sum_{m\sim q^3}\lambda(m)K(m)\delta(n,m\ell).
\end{split}
\end{equation*}
We first use a conductor-decreasing trick to write 
\begin{equation*}
\begin{split}
\delta(n,m\ell)=\delta_{q|n-m\ell}\cdot \delta\left(\frac{n-m\ell}{q},0\right).
\end{split}
\end{equation*}
Expressing $\delta_{q|n-m\ell}=\frac{1}{q}\sum_{u(q)}e\left(\frac{(n-m\ell)u}{q}\right)$ and using the delta symbol by Duke--Friedlander--Iwaniec \cite{DFI1.5} to detect the second one:
$$\delta\left(\frac{n-m\ell}{q},0\right)\approx \frac{1}C\sum_{c\sim C}\frac{1}{c}\sumstar_{u(c)}e\left(\frac{u(n-m\ell)}{cq}\right),$$
we have the following approximation
$$\delta(n,m\ell)\approx \frac{1}{qC}\sum_{c\sim C}\frac{1}{c}\sumstar_{u(cq)}e\left(\frac{u(n-m\ell)}{cq}\right).$$
Here $C$ is a large parameter which we will choose as 
$$C=\left(\frac{q^3L}{q}\right)^{1/2}=qL^{1/2}.$$
Now plugging this approximation for $\delta(n,m\ell)$ in, we can write our original sum as
\begin{equation*}
\begin{split}
\mathcal{S}\approx &\frac{\sqrt{q^3L}\sqrt{q^3}}{LCq} \sum_{c\sim C}\frac{1}{c}\sum_{\ell\sim L}\ov{\lambda(1,\ell)}\sumstar_{u(cq)}\sum_{n\sim q^3L}\frac{\lambda(1,n)}{\sqrt{n}}e\left(\frac{un}{cq}\right)\sum_{m\sim q^3}\frac{\lambda(m)}{\sqrt{m}}K(m)e\left(\frac{-um\ell}{cq}\right).
\end{split}
\end{equation*}
We use Fourier inversion to separate the $m$-variable from $K(m)$ and rewrite the $m$-sum above as
$$\frac{1}{q^{1/2}}\sum_{b\mods q}\widehat{K}(b)\sum_{m\sim q^3}\frac{\lambda(m)}{\sqrt{m}}e\left(\frac{-(bc+u\ell)m}{cq}\right).$$
We now use Voronoi summation to dualize the $n$- and $m$-variables respectively, getting
$$\sum_{n\sim q^3L}\frac{\lambda(1,n)}{\sqrt{n}}e\left(\frac{un}{cq}\right)\approx \sum_{n\sim \frac{(cq)^3}{q^3L}\approx q^3L^{1/2}}\frac{\lambda(n,1)}{\sqrt{n}}\Kl_2(\bar{u}n;cq);$$
$$\sum_{m\sim q^3}\frac{\lambda(m)}{\sqrt{m}}e\left(\frac{-(bc+u\ell)m}{cq}\right)\approx \sum_{m\sim \frac{(cq)^2}{q^3}\approx qL}\frac{\ov{\lambda(m)}}{\sqrt{m}}e\left(\frac{\ov{bc+u\ell}\,m}{cq}\right).$$
Therefore, after applying the two Voronoi summations, we arrive at the following dual sum
\begin{equation*}
\begin{split}
\mathcal{S}\approx &\frac{\sqrt{q^3L}}{LCq}\frac{q^2}{C} \sum_{c\sim C}\frac{1}{c}\sum_{\ell\sim L}\ov{\lambda(1,\ell)}\sum_{n\sim q^3L^{1/2}}\frac{\lambda(n,1)}{\sqrt{n}}\sum_{m\sim qL}\ov{\lambda(m)}\\
&\quad\quad\quad \frac{1}{q^{1/2}}\sumstar_{u(cq)}\sum_{b\mods q}\widehat{K}(b)\Kl_2(\bar{u}n;cq)e\left(\frac{\ov{bc+u\ell}\,m}{cq}\right).
\end{split}
\end{equation*}
By splitting the sum modulo $cq$ into a product of sums modulo $c$ and $q$ respectively, the second line above can be written as 
\begin{equation*}
\begin{split}
&\sumstar_{u_1(c)}\Kl_2(\bar{q}^2\ov{u_1}n;c)e\left(\frac{\ov{u_1\ell q}m}{c}\right)\cdot \frac{1}{q^{1/2}}\sum_{b\mods q}\widehat{K}(b)\sumstar_{u(q)}\Kl_2(\bar{c}^2\bar{u}n;q)e\left(\frac{\ov{bc+u\ell}\bar{c}m}{q}\right)\\
\approx &\sqrt{c}e\left(\frac{n\bar{q}\ell\ov{m}}{c}\right)\sumstar_{u(q)}\Kl_2(\bar{c}^3\bar{u}n;q)L_{\bar{c}^2m,\ell}(u;q),
\end{split}
\end{equation*}
where
$$L_{\alpha,\beta}(u;q):=\frac{1}{q^{1/2}}\sum_{b\mods q}\widehat{K}(b)e\left(\frac{\alpha\, \ov{b+\beta u}}{q}\right).$$
Hence
\begin{equation*}
\begin{split}
\mathcal{S}\approx &\frac{q^{5/2}}{L^{1/2}C^{3/2}} \sum_{n\sim q^3L^{1/2}}\frac{\lambda(n,1)}{\sqrt{n}}\sum_{c\sim C}\frac{1}{c}\sum_{\ell\sim L}\ov{\lambda(1,\ell)}\sum_{m\sim qL}\ov{\lambda(m)}e\left(\frac{n\bar{q}\ell\ov{m}}{c}\right)\sumstar_{u(q)}\Kl_2(\bar{c}^3\bar{u}n;q)L_{\bar{c}^2m,\ell}(u;q).
\end{split}
\end{equation*}
We next apply Cauchy--Schwarz inequality to remove the $\GL_3$ coefficients while keeping the $c$-sum ``inside":
\begin{multline*}
\mathcal{S}\ll\frac{q^{5/2}}{L^{1/2}C^{3/2}} \left(\sum_{n\sim q^3L^{1/2}}\frac{|\lambda(n,1)|^2}{n}\right)^{1/2}\left(\sum_{n\sim q^3L^{1/2}}\big|\sum_{c\sim C}\sum_{\ell\sim L}\sum_{m\sim qL}\sumstar_{u(q)} (...)\big|^2\right)^{1/2}\\
:=\frac{q^{5/2}}{L^{1/2}C^{3/2}}  \Omega^{1/2}.
\end{multline*}
\begin{remark}The $n$-variable has size $q^3L^{1/2}$ and was originally weighted by the $\GL_3$ coefficients $\lambda(n,1)$ times a periodic arithmetic function of modulus $qc$. After applying Cauchy--Schwarz the resulting periodic functions have moduli $qcc'\approx q^3L$ which is not much bigger than the size of $n$ which is now smooth and therefore one can expect to be able to analyse this sum further.
 \end{remark}
We continue to analyse the sum $\Omega$. Opening the square and switching the order of summations, we arrive at 
\begin{equation*}
\begin{split}
\Omega=&\sum_{c\sim C}\frac{1}{c}\sum_{c'\sim C}\frac{1}{c'}\sum_{\ell\sim L}\ov{\lambda(1,\ell)}\sum_{\ell'\sim L}\lambda(1,\ell')\sum_{m\sim qL}\ov{\lambda(m)}\sum_{m'\sim qL}\lambda(m')\\
&\sumstar_{u(q)}L_{\bar{c}^2m,\ell}(u;q)\sumstar_{u'(q)}\ov{L_{\ov{c'}^2m',\ell'}(u';q)}\sum_{n\sim q^3L^{1/2}}\Kl_2(\bar{c}^3\bar{u}n;q)\Kl_2(\ov{c'}^3\ov{u'}n;q)e\left(\frac{n\bar{q}\ell\ov{m}}{c}\right)e\left(\frac{-n\bar{q}\ell'\ov{m'}}{c'}\right).
\end{split}
\end{equation*}

We apply Poisson summation to the $n$-sum; after computing the resulting Fourier transform and several manipulations we obtain 
\begin{equation*}
\begin{split}
\Omega\approx&\sum_{c\sim C}\frac{1}{c}\sum_{c'\sim C}\frac{1}{c'}\sum_{\ell\sim L}\ov{\lambda(1,\ell)}\sum_{\ell'\sim L}\lambda(1,\ell')\sum_{m\sim qL}\ov{\lambda(m)}\sum_{m'\sim qL}\lambda(m')\\
&\quad  q^3L^{1/2}\sum_{|n|<{L}^{1/2}}\sum_{v(q)}Z(v)\, \ov{Z'(v-n\ov{cc'})}\cdot \delta_{n+\ell\ov mc'-\ell'\ov{m'}c\equiv 0 \mods {cc'}},
\end{split}
\end{equation*}
where
$$Z(v)=Z_{\alpha,\beta,\gamma}(v):=\frac{1}{q^{1/2}}\sum_{x\in\Fqt}\Kl_2(\beta\gamma x;q)K(xv)\Kl_2(\alpha xv;q)$$
with $(\alpha,\beta,\gamma)=(\ov c^2m,\ell,\ov c^3)$ and $Z'(v)$ is defined likewise with $(\alpha',\beta',\gamma')=(\ov{c'}^2m',\ell',\ov{c'}^3)$.


We consider the cases where $n=0$ and $n\neq 0$ separately and denote their contribution to $\Omega$ by $\Omega_0$ and $\Omega_{\neq}$ respectively.

For $n=0$, one has $c=c'$, and since the sheaf $\mcF$ is good, we have by  Proposition \ref{sqrootcancel},
$$\sum_{v}Z(v)\ov{Z'(v)}\ll q\cdot \delta_{\ell m'\equiv \ell'm \mods q}+q^{1/2};$$
this gives
\begin{equation*}
\begin{split}
\Omega_0=&\sum_{c\sim C}\frac{1}{c^2}\sum_{\ell\sim L}\ov{\lambda(1,\ell)}\sum_{\ell'\sim L}\lambda(1,\ell')\sum_{m\sim qL}\ov{\lambda(m)}\sum_{m'\sim qL}\lambda(m')\\
&\ \ \ \ \ \ \times q^3L^{1/2}(q\cdot \delta_{\ell m'\equiv \ell'm \mods q}+q^{1/2}) \delta_{\ell m'\equiv \ell'm \mods {c}}\\
\ll&\frac{q^5L^{5/2}}{C}+\frac{q^{11/2}L^{9/2}}{C^2},
\end{split}
\end{equation*}
whose corresponding contribution to $\mathcal{S}$ is
$$\mathcal{S}_0\ll \frac{q^3}{L^{1/4}}+q^{11/4}L^{1/2}.$$
\begin{remark}This bound is admissible is long as $L$ is a (not too big) positive power of $q$ and corresponds  roughly to the contribution of the diagonal term $(c,\ell,m,u)=(c',\ell',m',u')$ in $\Omega$:
\begin{equation*}
\begin{split}
\mathcal{S}_{00}\ll&\frac{q^{5/2}}{L^{1/2}C^{3/2}} \left(q^3L^{1/2}\cdot\frac{1}{C}\cdot L\cdot qL\cdot q \right)^{1/2}\ll \frac{q^3}{L^{1/4}}.
\end{split}
\end{equation*}
This saving was the whole point of keeping the $c$-sum inside in the application of Cauchy--Schwarz.
This also shows how the additional parameter $\ell$ plays a role in this argument: without introducing the $\ell$-sum at the very beginning we would have been failed to beat the convexity bound $O(q^3)$ for the diagonal contribution. This trick is reminiscent of the amplification technique which seems to be used first by Munshi in \cite{Munshi16} (see also \cite{HN, Lin, KLMS,sharma}).
\end{remark}

For the case $n\neq 0$, we use again Proposition \ref{sqrootcancel},
$$\sum_{v(q)}Z(v)\ov{Z'(v-\delta)}\ll q^{1/2},\ \delta\not=0\mods q,$$
to obtain
\begin{equation*}
\begin{split}
\Omega_{\neq}=&\sum_{c\sim C}\frac{1}{c}\sum_{c'\sim C}\frac{1}{c'}\sum_{\ell\sim L}\ov{\lambda(1,\ell)}\sum_{\ell'\sim L}\lambda(1,\ell')\sum_{m\sim qL}\ov{\lambda(m)}\sum_{m'\sim qL}\lambda(m')\\
&\ \ \ \ \times q^3L^{1/2}\sum_{0\neq |n|<L^{1/2}}q^{1/2}\cdot \delta_{n+\ell\ov mc'-\ell'\ov{m'}c\equiv 0 \mods {cc'}}\\
\ll& \frac{q^{11/2}L^5}{C^2}+\frac{q^{9/2}L^4}{C}.
\end{split}
\end{equation*}
Correspondingly, such terms contribute to $\mathcal{S}$
$$\mathcal{S}_{\neq}\ll q^{11/4}L^{3/4}.$$

Therefore we have the following estimate for the sum of interest
$$\mathcal{S}\ll \frac{q^3}{L^{1/4}}+q^{11/4}L^{1/2}+q^{11/4}L^{3/4}.$$
By choosing $L$ appropriately, say $L=q^{1/4}$ we have
$$\mathcal{S}\ll q^{3-1/16},$$
beating the convexity bound $\mathcal{S}=O(q^{3+o(1)})$.

Key to this argument is the bound for the correlation sums
$$\sum_{v}Z(v)\ov{Z'(v-\delta)}\ll \delta_{\delta=0}\delta_{\alpha/\alpha'\equiv \beta\gamma/(\beta'\gamma') \mods q}q+q^{1/2}$$
which follows from Proposition \ref{sqrootcancel}. This proposition is proven  by interpreting the functions
$$v\mapsto Z(v), {Z'(v)}$$
as trace functions of $\ell$-adic sheaves $\mcZ,\mcZ'$ on the affine line $\Aa^1_\Fq$ and by using methods from $\ell$-adic cohomology. As the expression for $Z$ suggests, the underlying sheaf $\mcZ$ is the geometric convolution of the tensor product
$$\mcK=(\mcF \otimes [\times \alpha]^*\HYPK_2 ) \hbox{ with }\mcL=[y\ra\beta\gamma/y]^*\HYPK_2$$
(and likewise for $\mcZ'$). First we verify that if (MO) is  satisfied, then $\mcK$ and hence $\mcZ$ are geometrically irreducible. It remains to show that if (SL) is also satisfied, $\mcZ$ and $[-\delta]\mcZ'$ are not geometrically isomorphic if either $\delta\not=0$, or $\alpha/\alpha'\not=\beta\gamma/(\beta'\gamma')\mods q$. To do this, we analyse carefully the possible singularities of $\mcZ$ (at most at $\infty$, $0$ and $\beta\gamma/\alpha$) and determine their nature (unipotent or not) at $0,\beta\gamma/\alpha$ using Deligne's semi-continuity theorem and the theory of local convolution due to Rojas-Le\'{o}n. This last point is the most delicate part of the geometric argument.

This general approach is entirely different from the path taken by Sharma in the special case $K=\chi$ a multiplicative character \cite{sharma}: Sharma exploits the multiplicativity of $\chi$ and reduces the problem to checking the non-degeneracy criterion of Adolphson--Sperber for an explicit exponential sum in $9$ variables. For this special case, a third route is possible: the sheaf $\mcZ$ is an hypergeometric sheaf whose local and global properties were studied in depth by Katz in \cite{ESDE}; the estimation of the correlation sums above can then be deduced from these properties and the general bounds for correlations sums of trace functions like \cite{FKMS}.

\subsection*{Acknowledgements} We are greatful to V. Blomer for his comments and suggestions on the paper and to the anonymous referee for his/her careful reading, numerous helpful comments and for pointing out some errors and inaccuracies. A good amount of this work was carried out during the {\em Second Symposium on Analytic Number Theory} which took place in Cetraro in July 2019 and we would like to thank D. Bazzanella, S. Bettin and A. Perelli for their excellent organisation and their choice of the beautiful location.
 Finally we would like to thank heartily E. Kowalski for decisive discussions and numerous encouragements during the various phases of this project.

\section{Background material regarding automorphic forms}
In the sequel we will denote by $V$ or $W$ some smooth functions $V,W\in\mcC^{\infty}_c([1,2[)$ satisfying for all $i\geq 0$
\begin{equation}\label{eqtestfct}
V^{(i)}(x)\ll Z^i	
\end{equation}
for some parameter $Z$ satisfying $1\leq Z\leq q$.

\subsection{Bounds for Hecke eigenvalues}
We recall the Hecke relation for $\GL_3$ Hecke eigenvalues implies that
\begin{equation*}
\lambda(r,n)=\sum_{d|(r,n)}\mu(d)\lambda(r/d,1)\lambda(1,n/d)
\end{equation*}
and 
\begin{equation}
 \label{eq-GL3-Hecke}
\lambda(1,m)\lambda(r,n)=\mathop{\sum_{d_0d_1d_2=m}}_{d_1|r,d_2|n}\lambda(rd_2/d_1, nd_0/d_2)
\end{equation}
and we have the following individual bounds
\begin{equation}
 \label{eq-Kim-Sarnak}
\lambda(r,n)\ll (rn)^{\theta_3+o(1)},
\end{equation}
where $\theta_3=5/14$ according to Kim--Sarnak (see \cite{KimSar}).

We also have the following well-known Rankin--Selberg estimates for $\GL_2$ and $\GL_3$ Hecke eigenvalues (see ~\cite[Sec. 12.1]{Goldfeld})
\begin{equation}
  \label{eq:RS}
  \sum_{1\leq n\leq X}|\lambda_f(n)|^2\ll_{f} X,
\end{equation}
and
\begin{equation*}
  \label{eq-RS1.5}
  \sum_{1\leq m^2n\leq X}|\lambda(n,m)|^2\ll_{\vphi} X,
\end{equation*}
from which and the Hecke relation, one can also derive the bound (see ~\cite[Lemma 2]{Munshi16})
\begin{equation}
  \label{eq-RS2}
  \sum_{1\leq m^2n\leq X}|\lambda(n,m)|^2m\ll_{\vphi, \eps} X^{1+\eps},
\end{equation}
for any $\eps>0$.


\subsection{Summation formulas}
Let $f$ be a $\GL_2$ cuspform of level $1$ and of weight $k$ (if $f$ is holomorphic) or of Laplace eigenvalue $1/4+r_f^2$ (if $f$ is Maass).

The $\GL_2$ Voronoi summation formula states:
\begin{proposition}\label{vorGL2}  Let $V$ be a test function as in \eqref{eqtestfct}. Let $X>0$, $c,u\geq 1$ be integers such that $(u,c)=1$. We have
$$\sum_{m=1}^\infty\lambda_f(m)e\left(\frac{um}{c}\right)V\left(\frac{m}{X}\right)=\frac{X}{c}\sum_{\pm}\sum_{m=1}^\infty \lambda_f(m)e\left(\mp\frac{\bar{u}m}{c}\right)\mathcal{V}^{\pm}\left(\frac{mX}{c^2}\right),$$
where
$$\mathcal{V}^{\pm}(y)=\int_0^{\infty} V(x)\mcJ^\pm_{f}(4\pi \sqrt{xy}){\mathrm{d}x}$$
with
$$\mcJ^+_{f}(x)=2\pi i^k J_{k-1}(x),\ \mcJ^-_{f}(x)=0\hbox{ if $f$ is holomorphic};$$
and  
\begin{equation*}\mcJ^+_{f}(x)=\frac{-\pi}{\sin \pi ir_f}(J_{2ir_f}(x)-J_{-2ir_f}(x)),\quad   \mcJ^-_{f}(x)=4\eps_f \cosh(\pi r_f)K_{2ir_f}(x)\, \quad \hbox{if $f$ is Maass}.
\end{equation*}	
\end{proposition}

We also recall the $\GL_3$ Voronoi summation formula first proven by Miller and Schmid (\cite{MS}):
\begin{proposition}\label{vorGL3}  Let $W$ be a test function as in \eqref{eqtestfct}. Let $X>0$, $c,u,r\geq 1$ be integers such that $(u,c)=1$. We have
\begin{equation*}
\begin{split}
\sum_{n=1}^\infty \lambda(r,n)e\left(\frac{un}{c}\right)W\left(\frac{n}{N}\right)=c^{3/2}\sum_{\pm}\sum_{n_1|rc}\sum_{n=1}^{\infty}\frac{\lambda(n,n_1)}{n_1n}\frac{S\left(r\bar{u},\pm n;\frac{rc}{n_1}\right)}{c^{1/2}}\mathcal{W}_{\pm}\left(\frac{nn_1^2}{c^3r/N}\right),
\end{split}\end{equation*}
where
$$S(m,n;c)=\sumstar_{x(c)}e\left(\frac{mx+n\bar{x}}{c}\right)$$
is the Kloosterman sum and
  where 
  \begin{equation}\label{eq-GL3transform}
  \mcW_{\sigma}(x)= \frac{1}{2\pi i}\int_{(1)}x^{-s}\mcG^{\sigma} (s+1)
  \Bigl(\int_{0}^{+\infty}W(y)y^{-s}\frac{dy}{y}\Bigr)ds.	
  \end{equation}

\end{proposition}
For the precise definition of the kernel function $\mcG^{\sigma} (s)$ we refer the reader to \cite[(4.6)]{BloKha}.

We recall the following lemma (see \cite[Lemma 4.3]{KLMS})
\begin{lemma}\label{bounds-for-V}
  Let $\sigma\in\{-1,1\}$. For any $j\geq 0$, any $A\geq 1$ and
  any~$\eps>0$, we have
$$
x^j\mcW_{\sigma}^{(j)}(x)\ll \min\Bigl( Z^{j+1}x^{1-\theta_3-\eps},
Z^{j+5/2+\eps}\Bigl(\frac{Z^3}{x}\Bigr)^A\Bigr)
$$
for~$x>0$, where the implied constant depends on~$(j,A,\eps)$.
Moreover, for $x\geq 1$, we have
$$
x^j\mcW_{\sigma}^{(j)}(x)\ll x^{2/3}\min(Z^j, x^{j/3})
$$
where the implied constant depends on~$j$.
\end{lemma}
The integral transform $\mcW_\pm$ admits another representation as Hankel transform of some $\GL_3$-Bessel functions as can be seen by switching the $y$ and $s$ integrals in \eqref{eq-GL3transform} which we record here (see \cite[Lemma 3.2]{Lin})
\begin{equation}\label{GL3Hankel}
\mcW_\pm(x)=x\int_0^\infty W(y)J_{\vphi,\pm}(xy)dy.
\end{equation}
Here
$J_{\vphi,\pm}(x)$ is some $\GL_3$-Bessel function satisfying the following properties.
\begin{lemma}\label{lem-Bessel3}
(1). Let $\rho>\max\{-\Re \mu_{\vphi,1}, -\Re \mu_{\vphi,2}, -\Re \mu_{\vphi,3} \}$. For $x\ll 1$, we have 
\begin{equation}\label{Bessel3-small-x}
x^jJ^{(j)}_{\vphi,\pm}(x)\ll_{\mu_{\vphi,1},\mu_{\vphi,2},\mu_{\vphi,3}, \rho, j} x^{-\rho}.
\end{equation}

(2). Let $K\geq 0$ be a fixed nonnegative integer. For $x>0$, we may write
\begin{equation}\label{Bessel3-large-x}
J_{\vphi,\pm}( x^3)=\frac{e(\pm 3x)}{x}W_\vphi^{\pm}(x)+E_\vphi^{\pm}(x),
\end{equation}
where $W_\vphi^{\pm}(x)$ and $E_\vphi^{\pm}(x)$ are real-analytic functions on $(0,\infty)$ satisfying 
\begin{equation*}\label{W_F}
W_\vphi^{\pm}(x)=\sum_{m=0}^{K-1}B^{\pm}_m(\vphi) x^{-m}+O_{K,\mu_{\vphi,1},\mu_{\vphi,2},\mu_{\vphi,3}}\left(x^{-K}\right),
\end{equation*}
and
\begin{equation*}
E_\vphi^{\pm,(j)}(x)\ll_{\mu_{\vphi,1},\mu_{\vphi,2},\mu_{\vphi,3}, j} \frac{\exp(-3\sqrt{3}\pi x)}{x},
\end{equation*}
for $x\gg_{\mu_{\vphi,1},\mu_{\vphi,2},\mu_{\vphi,3}} 1$, where $B^{\pm}_m(\vphi)$ are constants depending on $\mu_{\vphi,1}$, $\mu_{\vphi,2}$ and $\mu_{\vphi,3}$.
\end{lemma}
\begin{proof}
See \cite[Theorem 14.1]{Qi}.
\end{proof}

Next we recall the Duke--Friedlander--Iwaniec delta symbol method \cite{DFI1.5} in a version given by Heath-Brown \cite{HB}.
\begin{lemma}\label{HB-delta}
For any $C>1$, there is a positive constant $\eta_C$ and a smooth function $h(x,y)$ defined on $(0,\infty)\times (-\infty,\infty)$ such that 
$$\delta_{n=0}=\frac{\eta_C}{C}\sum_{c=1}^{\infty}\frac{1}{c}\sumstar_{a(c)}e\left(\frac{an}{c}\right)h\left(\frac{c}{C},\frac{n}{C^2}\right).$$
Here the constant $\eta_C$ satisfies $$\eta_C=1+O_A(C^{-A})$$ for any $A>0$, and $h(x,{y})$ is a smooth function vanishing unless 
$x\leq \max(1,2|y|)$ and whose derivatives satisfy 
$$x^i\frac{\partial^{i}}{\partial^{i}x}h(x,{y})\ll_{i}1\, \text{and}\, \frac{\partial}{\partial y}h(x,{y})=0,$$
for $x\leq 1$ and $|y|\leq x/2$, and
\begin{equation}\label{hderivatives}
\begin{split}
	x^iy^j\frac{\partial^{i+j}}{\partial^{i}x\partial^{j}y}h(x,{y})\ll_{i,j}1
\end{split}
\end{equation} 
for $|y|>x/2$.
\end{lemma}

\section{First transformations} 
From now on we will assume that $Z$ satisfies
$$1\leq Z\leq q$$
(otherwise the trivial bound is stronger that the bound claimed in Theorem \ref{thmStVbound}). 
We have
$$S^t_V(K,X):=\sum_{r,n}\lambda(r,n)\lf(n)K(nr^2)V(\frac{nr^2}{X})=\sum_{r\geq 1}S_{V,r}(K,X/r^2).$$
where
\begin{equation}\label{Sdef}
S_{V,r}(K,X):=\sum_{n=1}^{\infty}\lambda(r,n)\lambda_f(n)K(nr^2)V(\frac{n}{X}).	
\end{equation}
By applying the Rankin--Selberg estimates \eqref{eq:RS}, \eqref{eq-RS2} and the bound \eqref{eq-Kim-Sarnak}, we have
\begin{equation}\label{large-R}
\sum_{r\geq R}S_{V,r}(K,X/r^2)\ll q^{o(1)}R^{\theta_3}X/R,
\end{equation}
where $\theta_3=5/14$; see, say \cite[(4.2)]{Lin-Sun} for the details. It will suffice to bound $S_{V,r}(K,X)$ for $r\leq R=q^\rho$ for some $\rho>0$ to be chosen; see \eqref{choice-of-R}. In the sequel we assume that $$r\leq R;$$ moreover, as we will see, $R<q$, so that $(r,q)=1$.

Let $L\geq 1$ be some parameter and let $\rmL$ be the set of primes in the interval $[L,2L[$. We first write
$$S_{V,r}(K,X)=\frac{1}{L^\star}\sum_{\ell\in\rmL}\ov{\lambda(1,\ell)}\sum_n\lambda(1,\ell)\lambda(r,n)\lf(n)K(nr^2)V(\frac{n}X)$$
where $$L^\star:=\sum_{\ell\in\rmL}|\lambda(1,\ell)|^2 \simeq_{\vphi} \frac{L}{\log L}$$
by the prime number theorem for automorphic forms (\cite[Lemma 5.1]{LWY}).

In the sequel we will assume that $L$ satisfies
\begin{equation}\label{Lbound}
	R<L
\end{equation}
so that it is guaranteed that any $\ell\in [L,2L[$ is coprime with $r$.

By the Hecke relation \eqref{eq-GL3-Hecke} we have (since $(\ell,r)=1$ by the assumption $L>R$) 
$$\lambda(1,\ell)\lambda(r,n)=\lambda(r,\ell n)+\delta_{\ell|n}\lambda(r\ell,n/\ell).$$
The contribution to $S_{V,r}(K,X)$ of the second term is trivially bounded by
\begin{equation}\label{SVKsecond}\frac{1}{L^\star}\sum_{\ell\in\rmL}|\lambda(1,\ell)|^2\sum_{n\sim X/\ell}|\lambda(r,n)|\left(|\lf(n)\lf(\ell)|+\delta_{\ell|n}|\lf(n/\ell)|\right)\|K\|_\infty\leq \qoo \frac{r^{\theta_3}L^{\theta_2}X}{L}	
\end{equation}
where $\theta_3=5/14$ and $\theta_2=7/64$.
Therefore we obtain
$$S_{V,r}(K,X)=\frac{1}{L^\star}\sum_{\ell\in\rmL}\ov{\lambda(1,\ell)}\sum_{m=1}^{\infty}\lambda(r,m\ell)\lf(m)K(mr^2)V\left(\frac{m}X\right)+O\left(\qoo \frac{r^{\theta_3}L^{\theta_2}X}{L}\right).$$

Now we use the delta method to separate the coefficients $\lambda(r,m\ell)$ and $\lf(m)K(mr^2)$. To prepare for later manipulations, we first introduce $U$, a smooth function  supported in $(1/100,100)$ and satisfying $U^{(i)}(x)\ll 1$ for $i\geq 0$ and $U(x)=1$ for $x\in [1,2]$. 

Let $v=q^{\varepsilon}Z$ where $\eps>0$ is fixed but to be chosen as small as we need.
Then we can rewrite $S_{V,r}(K,X)$ as 
\begin{multline*}
S_{V,r}(K,X)=\frac{1}{L^\star}\sum_{\ell\in\rmL}\ov{\lambda(1,\ell)}\sum_{n=1}^{\infty}\lambda(r,n)\sum_{m=1}^{\infty}\lf(m)K(mr^2)\left(\frac{n}{m\ell}\right)^{iv}\delta_{n=m\ell}U\left(\frac{n}{X\ell}\right)V\left(\frac{m}X\right)\\ +O\left(\qoo \frac{r^{\theta_3}L^{\theta_2}X}{L}\right).	
\end{multline*}

Using a Mellin transform on $U$ we can replace -- up to a factor $\qoo$ and up to changing the definition of $U$-- the expression $U\left(\frac{n}{X\ell}\right)$ by $U\left(\frac{n}{XL}\right)$.

Therefore from now on we will consider the sum
\begin{equation}
     S'_{V,r}(K,X)=\frac{1}{L^\star}\sum_{\ell\in\rmL}\ov{\lambda(1,\ell)}\sum_{n=1}^{\infty}\lambda(r,n)\sum_{m=1}^{\infty}\lf(m)K(mr^2)\left(\frac{n}{m\ell}\right)^{iv}\delta_{n=m\ell}U\left(\frac{n}{XL}\right)V\left(\frac{m}X\right).    \label{SVXfirst}
\end{equation}

\begin{remark}\label{remvlocalisation}
	The introduction of the parameter $v=q^{\varepsilon}Z$ will be useful later, when we try to localize the range of some variables; see Lemma \ref{lem-vlocalize} and \eqref{MNdef}. See also \cite{AHLS} (and Remark 4.2 there) for a similar trick.
\end{remark}

We proceed as in \cite{sharma} and follow Holowinsky, Munshi and Qi \cite{HMQ} to rewrite the $\delta_{n-\ell m}$ in a more analytic form using the delta symbol method in Lemma \ref{HB-delta}: for any $n\in\Zz$ such that $|n|\leq 4XL$ we have
\begin{eqnarray*}
\delta_{n=0}=\delta_{q|n}\delta_{n/{q}=0}&=& \frac{1}{q}\sum_{u(q)}e(\frac{un}{q})\frac{1}{C}\sum_{c\leq  2C}\frac{1}{c}\sumstar_{a(c)}e\left(\frac{an/q}{c}\right)h\left(\frac{c}{C},\frac{n}{C^2q}\right)+O_A(C^{-A})\\
   & =&\frac{1}{C}\sum_{c\leq  2C}\frac{1}{cq}\sumsumstar_{u(q),a(c)}e\left(n\frac{a+uc}{cq}\right)h\left(\frac{c}{C},\frac{n}{C^2q}\right)+O_A(C^{-A})\\
   &=&\frac{1}{C}\sum_\stacksum{c\leq  2C}{(c,q)=1}\frac{1}{cq}\sumsumstar_{u(q),a(c)}e\left(n\frac{aq+uc}{cq}\right)h\left(\frac{c}{C},\frac{n}{C^2q}\right)\\
&+&\frac{1}{C}\sum_{c\leq  2C/q}\frac{1}{cq^2}\sumsumstar_{u(q),a(cq)}e\left(n\frac{a+ucq}{cq^2}\right)h\left(\frac{cq}{C},\frac{n}{C^2q}\right)+O_A(C^{-A}).
\end{eqnarray*}
Here we choose 
\begin{equation}\label{size-of-C}
2C= \left(\frac{XL}{q}\right)^{\frac{1}{2}}.
\end{equation}

Observe that in the first sum of the last expression, as $a$ varies over a set of representatives of the residue classes modulo $c$ (prime to $c$) and $u$ varies over a set of representatives of the residue classes modulo $q$, $aq+uc$ varies over a set of representatives of the residue classes modulo $cq$ prime to $c$.

Similarly in the second sum, the modulus  $c$ is $\leq 2(XL/q^3)^{1/2}<q$ and is therefore coprime with $q$. It follows that as $a$ varies over a set of representatives of the residue classes modulo $cq$ (prime to $cq$) and $u$ varies over a set of representatives of the residue classes modulo $q$, $a+ucq$ varies over a set of representatives of the residue classes modulo $cq^2$ prime to $cq^2$.

We can therefore rewrite

\begin{eqnarray}
\delta_{n=0}&=&\frac{1}{C}\sum_\stacksum{c\leq  2C}{(c,q)=1}\frac{1}{cq}\sumstar_{u(cq)}e\left(n\frac{u}{cq}\right)h\left(\frac{c}{C},\frac{n}{C^2q}\right)\nonumber\\
&+&\frac{1}{C}\sum_\stacksum{c\leq  2C}{(c,q)=1}\frac{1}{cq}\sumstar_{a(c)}e\left(n\frac{a}{c}\right)h\left(\frac{c}{C},\frac{n}{C^2q}\right)\label{eqdelta}\\
&+&
\frac{1}{C}\sum_{c\leq  2C/q}\frac{1}{cq^2}\sumstar_{u(cq^2)}e\left(n\frac{u}{cq^2}\right)h\left(\frac{cq}{C},\frac{n}{C^2q}\right)+O_A(C^{-A}).\nonumber
\end{eqnarray}

 \begin{remark}One reason for detecting the condition $n-m\ell$ in two such steps is that since we already have a trace function of period $q$, introducing an additive character modulo $q$ does not significantly increase the complexity. Therefore, we can detect the condition $n=m\ell \mods q$ at a reduced cost. After this, we can then detect $(n-m\ell)/q=0$ using the delta symbol.  So we are choosing a delta symbol that is in harmony with the modulus of our trace function (we thank the referee for pointing this out).
 
 The outcome of such an operation is that the parameter $C$ in \eqref{size-of-C} is reduced by a factor $q^{1/2}$ from the most natural choice $(XL)^{1/2}$. A similar reduction trick was used by Munshi in \cite{Munshi15}, and more recently in \cite{SubGL2GL3}.	
\end{remark}

We apply \eqref{eqdelta} to the difference $n-m\ell$ in \eqref{SVXfirst} and obtain
\begin{equation}\label{S'sumErr}
S'_{V,r}(K,X)=\mathrm{Main}+\mathrm{Err}_1+\mathrm{Err}_2+O_A(X^{-A})
\end{equation}
where
\begin{equation}
\begin{split}
       \mathrm{Main}=&\frac{1}{L^\star Cq}\sum_\stacksum{c\leq  2C}{(c,q)=1}\frac{1}{c}\sum_{\ell\in\rmL}\ov{\lambda(1,\ell)}\ell^{-iv}\sumstar_{u(cq)}\sum_{n=1}^{\infty}\lambda(r,n)e\left(n\frac{u}{cq}\right)n^{iv}U\left(\frac{n}{XL}\right)\\
        &\sum_{m=1}^{\infty}\lf(m)K(mr^2)e\left(\frac{-um\ell}{cq}\right)m^{-iv}V\left(\frac{m}X\right)h\left(\frac{c}{C},\frac{n-m\ell}{C^2q}\right),
\end{split}\end{equation}
\begin{equation}\label{eqErr2}
\begin{split}
        \mathrm{Err}_1=&\frac{1}{L^\star Cq}\sum_\stacksum{c\leq  2C}{(c,q)=1}\frac{1}{c}\sum_{\ell\in\rmL}\ov{\lambda(1,\ell)}\ell^{-iv}\sumstar_{a(c)}\sum_{n=1}^{\infty}\lambda(r,n)e\left(\frac{an}{c}\right)n^{iv}U\left(\frac{n}{XL}\right)\\
        &\sum_{m=1}^{\infty}\lf(m)K(mr^2)e\left(\frac{-am\ell}{c}\right)m^{-iv}V\left(\frac{m}X\right)h\left(\frac{c}{C},\frac{n-m\ell}{C^2q}\right),
\end{split}\end{equation}
and
\begin{equation}\label{eqErr3}
\begin{split}
        \mathrm{Err}_2=&\frac{1}{L^\star Cq^2}\sum_\stacksum{c\leq  2C/q}{(c,q)=1}\frac{1}{c}\sum_{\ell\in\rmL}\ov{\lambda(1,\ell)}\ell^{-iv}\sumstar_{a(cq^2)}\sum_{n=1}^{\infty}\lambda(r,n)e\left(\frac{an}{cq^2}\right)n^{iv}U\left(\frac{n}{XL}\right)\\
        &\sum_{m=1}^{\infty}\lf(m)K(mr^2)e\left(\frac{-am\ell}{cq^2}\right)m^{-iv}V\left(\frac{m}X\right)h\left(\frac{cq}{C},\frac{n-m\ell}{C^2q}\right).
\end{split}\end{equation}

As for the first term $\mathrm{Main}$, we may further restrict to the subsum satisfying $(c,\ell)=1$ and will bound the complementary subsum (corresponding to $\ell |c$) as an error term. That is, we further write 
$$\mathrm{Main}=\mathrm{Main}_0+\mathrm{Err}_3,$$ where
\begin{equation}\label{S'sum}
\begin{split}
       \mathrm{Main}_0=&\frac{1}{L^\star Cq}\sum_\stacksum{c\leq  2C}{(c,q)=1}\frac{1}{c}\sum_\stacksum{\ell\in\rmL}{(\ell,c)=1}\ov{\lambda(1,\ell)}\ell^{-iv}\sumstar_{u(cq)}\sum_{n=1}^{\infty}\lambda(r,n)e\left(n\frac{u}{cq}\right)n^{iv}U\left(\frac{n}{XL}\right)\\
        &\sum_{m=1}^{\infty}\lf(m)K(mr^2)e\left(\frac{-um\ell}{cq}\right)m^{-iv}V\left(\frac{m}X\right)h\left(\frac{c}{C},\frac{n-m\ell}{C^2q}\right),
\end{split}\end{equation}
and 
\begin{equation}\label{eqErr4}
\begin{split}
      \mathrm{Err}_3=&\frac{1}{L^\star Cq}\sum_{\ell\in\rmL}\ov{\lambda(1,\ell)}\ell^{-iv}\sum_\stacksum{c\leq  2C}{\ell |c, (c,q)=1}\frac{1}{c}\sumstar_{u(cq)}\sum_{n=1}^{\infty}\lambda(r,n)e\left(n\frac{u}{cq}\right)n^{iv}U\left(\frac{n}{XL}\right)\\
        &\sum_{m=1}^{\infty}\lf(m)K(mr^2)e\left(\frac{-um\ell}{cq}\right)m^{-iv}V\left(\frac{m}X\right)h\left(\frac{c}{C},\frac{n-m\ell}{C^2q}\right).
\end{split}\end{equation}
In the sequel, we focus our analysis on the term $\mathrm{Main}_0$ which is the hardest and is responsible for the final bound. The other three terms $\mathrm{Err}_1$, $\mathrm{Err}_2$ and $\mathrm{Err}_3$ are discussed briefly in \S \ref{Err23}. 

\subsection{Bounding $\mathrm{Main}_0$} 

To prepare for the application of Voronoi summation formula to the $m$-sum we write
$$K(mr^2)=\frac{1}{q^{1/2}}\sum_{b\mods q}\widehat{K}(b)e(\frac{-b mr^2}{q})=
\frac{1}{q^{1/2}}\sum_{b\mods q}\widehat{K}(b)e(\frac{-bcr^2m}{cq}),$$
where
\begin{equation}\label{fourierdef}
\widehat{K}(b)=\frac{1}{q^{1/2}}\sum_{a\in\Fq}K(a)e(\frac{ab}q)	
\end{equation}
denote the normalized Fourier transform of $K$.

We find that the term in \eqref{S'sum} can be rewritten as
\begin{equation}\label{S'sum-2}
\begin{split}
       \mathrm{Main}_0=&\frac{1}{L^\star Cq}\sum_\stacksum{c\leq  2C}{(c,q)=1}\frac{1}{c}\sum_\stacksum{\ell\in\rmL}{(\ell,c)=1}\ov{\lambda(1,\ell)}\ell^{-iv}\sumstar_{u(cq)}\sum_{n=1}^{\infty}\lambda(r,n)e\left(n\frac{u}{cq}\right)n^{iv}U\left(\frac{n}{XL}\right)\\
        &\frac{1}{q^{1/2}}\sum_{b\mods q}\what K(b)\sum_{m=1}^{\infty}\lf(m)e\left(\frac{-(bcr^2+u\ell)m}{cq}\right)m^{-iv}V\left(\frac{m}X\right)h\left(\frac{c}{C},\frac{n-m\ell}{C^2q}\right).
\end{split}\end{equation}
We can further assume that $(bcr^2+u\ell,cq)=1$, as otherwise we would have $(bcr^2+u\ell,q)=q$ (since we have already assumed $(c,\ell)=1$) and then by applying Proposition \ref{vorGL2} to such terms above,  we have
\begin{equation*}
\begin{split}
&\sum_{m=1}^{\infty}\lf(m)e\left(\frac{-\frac{bcr^2+u\ell}{q}m}{c}\right)m^{-iv}V\left(\frac{m}X\right)h\left(\frac{c}{C},\frac{n-m\ell}{C^2q}\right)\\
&=\frac{X}{c}\sum_{m=1}^{\infty}\ov{\lf(m)}e(\frac{\pm\ov{\frac{bcr^2+u\ell}{q}}m}{c})\widehat{\mathcal{V}}^{\pm}\left(n,\frac{mX}{c^2}\right),
\end{split}\end{equation*}
where
\begin{equation}\label{v-hat}
\begin{split}
\widehat{\mathcal{V}}^{\pm}\left(n,y\right)=\int_{\mathbb{R}}V(x)(Xx)^{-iv}h\left(\frac{c}{C},\frac{n-Xx\ell}{C^2q}\right)\mcJ^\pm_{f}(4\pi \sqrt{xy}){\mathrm{d}x}.
        \end{split}\end{equation}
        By integration by parts, one easily sees that $\widehat{\mathcal{V}}^{\pm}\left(n,y\right)\ll_A \left(\frac{Z+v}{\sqrt{y}}\right)^A$, for any $A\geq 0$. On the other hand with our choice of the $C$ in \eqref{size-of-C}, we have
        $$\frac{mX}{c^2}>\frac{mX}{C^2}=\frac{mq}{L}\gg q^{2\eta}Z^2$$
        for some $\eta>0$. Therefore $\widehat{\mathcal{V}}^{\pm}\left(n,\frac{mX}{c^2}\right)\ll (q^{-\eta})^A$, of arbitrarily small size. That is, the contribution from the terms with $(bcr^2+u\ell,cq)>1$ in \eqref{S'sum-2} is bounded above by $O_A(q^{-A})$.
        
Now assuming $(bcr^2+u\ell,cq)=1$ and applying Proposition \ref{vorGL2} to the $m$-sum in \eqref{S'sum-2}, we have
\begin{equation}\label{eqafter1stvoronoi}
\begin{split}
              \mathrm{Main}_0=&\frac{X}{L^\star Cq^2}\sum_{\pm}\sum_\stacksum{c\leq  2C}{(c,q)=1}\frac{1}{c^2}\sum_\stacksum{\ell\in\rmL}{(\ell,c)=1}\ov{\lambda(1,\ell)}\ell^{-iv}\frac{1}{q^{1/2}}\sumsumstar_\stacksum{b(q),u(cq)}{(bcr^2+u\ell,cq)=1}\what K(b)\\
        &\times\sum_{m=1}^{\infty}\ov{\lf(m)}e(\frac{\pm\ov{bcr^2+u\ell}m}{cq})\sum_{n=1}^{\infty}\lambda(r,n)e\left(n\frac{u}{cq}\right)n^{iv}U\left(\frac{n}{XL}\right)\widehat{\mathcal{V}}^{\pm}\left(n,\frac{mX}{c^2q^2}\right)+O(q^{-A}).
        \end{split}\end{equation}

We further apply Proposition \ref{vorGL3} to the $n$-sum above to obtain a sum of the form
\begin{multline}\label{eqaftervoronoi}
\frac{X}{L^\star Cq}\sum_{\pm\pm}\sum_\stacksum{c\leq  2C}{(c,q)=1}\frac{1}{c}\sum_\stacksum{\ell\in\rmL}{(\ell,c)=1}\ov{\lambda(1,\ell)}\ell^{-iv}\\\times\sum_{n_1|rcq}\sum_{m,n}\ov{\lf(m)}\frac{\lambda(n,n_1)}{nn_1}\mathcal{C}(m,n;\frac{rcq}{n_1})\mcW_{\pm\pm}\left(\frac{m}{c^2q^2/X},\frac{n_1^2n}{c^3q^3r/XL}\right)	
\end{multline}
where
$$\mathcal{C}(m,n;\frac{rcq}{n_1})=\frac{1}{q^{1/2}}\sumsumstar_\stacksum{b(q),u(cq)}{(bcr^2+u\ell,cq)=1}\what K(b)e\left(\frac{\pm\ov{bcr^2+u\ell}m}{cq}\right)S\left(r\ov u,\pm n;\frac{rcq}{n_1}\right)$$
and
\begin{equation}\label{double-weight-function}
\begin{split}
\mcW_{\pm\pm}(y,z)&=\int_{\mathbb{R}}V(x)(Xx)^{-iv}\mathcal{W}_{x,\pm}(z)\mcJ^\pm_{f}(4\pi \sqrt{xy}){\mathrm{d}x}\\
=&z\int_{\mathbb{R}}V(x)(Xx)^{-iv}\bigl(\int_{0}^{\infty}W_x(\xi)J_{\vphi,\pm}(z\xi){\mathrm{d}\xi}\bigr)\, \mcJ^\pm_{f}(4\pi \sqrt{xy}){\mathrm{d}x}
  \end{split}\end{equation}
by \eqref{GL3Hankel}. Here $$W_x(\xi):=(XL\xi)^{iv}U(\xi)h\left(\frac{c}{C},\frac{XL\xi-Xx\ell}{C^2q}\right).$$
\begin{lemma}\label{lem-vlocalize} Let $v=q^{\varepsilon}Z,\ \eps>0$ be the parameter introduced above Remark \ref{remvlocalisation}. For any $B\geq 1$  the function 
$$(y,z)\mapsto \mcW_{\pm\pm}(y,z)$$
is negligible for $y,z\geq q^{-B}$ unless 
\begin{equation}\label{yzphase}
y\asymp v^2,\ z\asymp v^3.	
\end{equation}
 Here negligible means that for any $A\geq 1$, and for $y,z\geq q^{-B}$ not satisfying \eqref{yzphase}, one has
$$\mcW_{\pm\pm}(y,z)\ll_{A,B,\eps}q^{-A}.$$ 
\end{lemma}
\proof
We recall from \eqref{v-hat} that
$$\widehat{\mathcal{V}}^{\pm}\left(n,y\right)=\int_{\mathbb{R}}V(x)(Xx)^{-iv}h\left(\frac{c}{C},\frac{n-Xx\ell}{C^2q}\right)\mcJ^\pm_{f}(4\pi \sqrt{xy}){\mathrm{d}x}
$$
for $n$ a real variable satisfying $n\asymp XL$ and that
$$\mcW_{\pm\pm}(y,z)=z\int_{0}^{\infty}(XL\xi)^{iv}U(\xi)\widehat{\mathcal{V}}^{\pm}\left(XL\xi ,y\right)J_{\vphi,\pm}(z\xi){\mathrm{d}\xi}$$

We consider the cases where $f$ is holomorphic (the case $f$ is a Maass cusp form would be similar). Then $\mcJ^-_{f}(x)=0$, and $\mcJ^+_{f}(x)=J_{k-1}(x)$ satisfies
\begin{equation}\label{derivative-of-J}
x^i J_{k-1}^{(i)}(x)\ll x^{k-1},
\end{equation}
for $x\ll 1$, and while for $x\gg 1$ we can write
\begin{equation}\label{expansion-of-J}
J_{k-1}(x)=\sum_{\pm}\frac{e^{\pm ix}}{\sqrt{x}}W_{k-1,\pm}(x),
\end{equation}
for some  $W_{k-1,\pm}(x)$ satisfying $x^iW_{k-1,\pm}^{(i)}(x)\ll 1$. We also have
$$x^i\frac{\partial^i}{\partial x^i}h\left(\frac{c}{C},\frac{n-Xx\ell}{C^2q}\right)\ll 1,$$
which follows from \eqref{hderivatives}. This together with \eqref{derivative-of-J} implies that if $y\ll 1$ then $x^iV^{(i)}_{1,y}(x)\ll Z^i$, where $$V_{1,y}(x):=V(x)X^{-iv}h\left(\frac{c}{C},\frac{n-Xx\ell}{C^2q}\right)J_{k-1}(4\pi \sqrt{xy}).$$ Therefore integration by parts implies that: if $y\ll 1$ we have
$$\widehat{\mathcal{V}}^{+}\left(n,y\right)=\int_{\mathbb{R}}V_{1,y}(x)x^{-iv}{\mathrm{d}x}\ll \left(\frac{Z}{v}\right)^A\ll \left(\frac{1}{q^{\varepsilon}}\right)^A,$$
for any $A\geq 0$. That is, for $y\ll 1$, $\widehat{\mathcal{V}}^{+}\left(n,y\right)$ is always negligibly small.

We now assume that $y\gg 1$. We use \eqref{expansion-of-J} to write
\begin{equation*}
\begin{split}
\widehat{\mathcal{V}}^{+}\left(n,y\right)=&\sum_{\pm}\int_{\mathbb{R}}X^{-iv}V(x)h\left(\frac{c}{C},\frac{n-Xx\ell}{C^2q}\right)\frac{1}{(4\pi\sqrt{xy})^{1/2}}W_{k-1,\pm}(4\pi\sqrt{xy})e\left(-\frac{v\log x}{2\pi}\pm2\sqrt{xy}\right){\mathrm{d}x}\\
:=&\sum_{\pm}\frac{1}{y^{1/4}}\int_{\mathbb{R}}V_{2,y}(x)e\left(-\frac{v\log x}{2\pi}\pm2\sqrt{xy}\right){\mathrm{d}x},
\end{split}
\end{equation*}
where $V_{2,y}(x)$ satisfies $x^iV^{(i)}_{2,y}(x)\ll Z^i$. Again, integration by parts implies that
\begin{equation*}
\begin{split}
\widehat{\mathcal{V}}^{+}\left(n,y\right)\ll \frac{1}{y^{1/4}}\left(\frac{Z}{|-v/2\pi\pm \sqrt{y}|}\right)^A,
\end{split}
\end{equation*}
for any $A\geq 0$, which shows that $\widehat{\mathcal{V}}^{+}\left(n,y\right)$ is negligibly small unless $|-v/2\pi\pm \sqrt{y}|\leq v$. Therefore we may assume that the $y$-variable satisfies
$$y\asymp v^2$$
(otherwise $\widehat{\mathcal{V}}^{+}\left(n,y\right)$ and $\mcW_{\pm\pm}(y,z)$ have negligible size for $n\asymp XL$ and $y,z\geq q^{-B}$.)

In this remaining range we evaluate $\mcW_{\pm\pm}(y,z)$. Recall from \eqref{double-weight-function}, we have
\begin{equation}\label{double-weight-function-2}
\begin{split}&\mcW_{\pm\pm}(y,z)=\int_{\mathbb{R}}V(x)(Xx)^{-iv} \mcJ^\pm_{f}(4\pi \sqrt{xy})\\
&\quad\quad\quad \times z\int_{0}^{\infty}(XL\xi)^{iv}U(\xi)h\left(\frac{c}{C},\frac{XL\xi-Xx\ell}{C^2q}\right)J_{\vphi,\pm}(z\xi){\mathrm{d}\xi} \,{\mathrm{d}x}.\\
\end{split}
\end{equation}
We consider the inner integral above. Recall here $U(\xi)$ is a test function satisfying \eqref{eqtestfct} with $Z_U=1$. If $z\ll 1$, then by using \eqref{Bessel3-small-x} and \eqref{hderivatives} we have $\xi^iW^{(i)}_{1,z}(\xi)\ll z^{-\rho}$, where $$W_{1,z}(\xi):=(XL)^{iv}U(\xi)h\left(\frac{c}{C},\frac{XL\xi-Xx\ell}{C^2q}\right)J_{\vphi,\pm}(z\xi).$$ Hence using integration by parts, the inner integral satisfies
\begin{equation*}
\int_{\mathbb{R}}W_{1,z}(\xi)\xi^{iv}{\mathrm{d}\xi} \ll z^{-\rho}v^{-A}\ll q^{-A\eps/2}
\end{equation*}
by taking $A$ sufficiently large. This implies that $\mcW_{\pm\pm}(y,z)$ is negligibly small when $q^{-B}\leq z\ll 1$.

Now we assume that $z\gg 1$ and use \eqref{Bessel3-large-x} to rewrite the inner integral over $\xi$ in \eqref{double-weight-function-2} as
\begin{gather*}
\sum_{\pm}\int_{0}^{\infty}(XL)^{iv}U(\xi)h\left(\frac{c}{C},\frac{XL\xi-Xx\ell}{C^2q}\right)\frac{W_{\vphi}^{\pm}((z\xi)^{1/3})}{(z\xi)^{1/3}}e\left(\frac{v\log \xi}{2\pi}\pm3(z\xi)^{1/3}\right){\mathrm{d}\xi}\\
:=\sum_{\pm}\frac{1}{z^{1/3}}\int_{0}^{\infty}W_{2,z}(\xi)e\left(\frac{v\log \xi}{2\pi}\pm3(z\xi)^{1/3}\right){\mathrm{d}\xi}
\end{gather*}
up to a negligible error term. Here $$W_{2,z}(\xi)=(XL)^{iv}\xi^{-1/3}U(\xi)h\left(\frac{c}{C},\frac{XL\xi-Xx\ell}{C^2q}\right)W_{\vphi}^{\pm}((z\xi)^{1/3})$$ satisfies $\xi^iW^{(i)}_{2,z}(\xi)\ll 1$. Applying integration by parts, the integral above is negligibly small unless 
$$z\asymp v^3.$$

\qed
\begin{remark} The above argument shows in fact that for $z\geq q^{-B}$ and $y>0$ one has
\begin{equation}\label{expansion-of-W}
\begin{split}&\mcW_{\pm\pm}(y,z)=z^{2/3}\int_{\mathbb{R}}V(x)(Xx)^{-iv} \mcJ^\pm_{f}(4\pi \sqrt{xy})\\
&\quad\quad\quad \times \int_{0}^{\infty}W_{2,z}(\xi)e\left(\frac{v\log \xi}{2\pi}-3(z\xi)^{1/3}\right){\mathrm{d}\xi} \,{\mathrm{d}x}+``\text{negligible\, error\, term}".\\
\end{split}
\end{equation}
	where the main term is also negligible unless possibly, when \eqref{yzphase} is satisfied.
	
\end{remark}

\section{The case $q\not| n_1$}
For the sum in \eqref{eqaftervoronoi}, we further split it into two subsums according to $(n_1,q)=1$ or not, and write 
$$ \mathrm{Main}_0= \mathrm{Main}_{00}+ \mathrm{Err}_4+O(q^{-A}),$$
where
\begin{multline}
\mathrm{Main}_{00}=\frac{X}{L^\star Cq}\sum_{\pm\pm}\sum_\stacksum{c\leq  2C}{(c,q)=1}\frac{1}{c}\sum_\stacksum{\ell\in\rmL}{(\ell,c)=1}\ov{\lambda(1,\ell)}\ell^{-iv}\sum_\stacksum{n_1|rc}{(n_1,q)=1}\sum_{m,n}\ov{\lf(m)}\\
\times\frac{\lambda(n,n_1)}{nn_1}\mathcal{C}(m,n;\frac{rcq}{n_1}) \mcW_{\pm\pm}\left(\frac{m}{c^2q^2/X},\frac{n_1^2n}{c^3q^3r/XL}\right)\label{n1qcoprime}
\end{multline}
and $\mathrm{Err}_4$ corresponds to the complementary sum where $q|n_1$. 
In \S \ref{q-divide-n1} we briefly analyse the contribution from $\mathrm{Err}_4$ (see \eqref{qdivn1} and \eqref{qdivn1final}).

We have
$$e(\frac{\pm\ov{bcr^2+u\ell}m}{cq})=e(\frac{\pm\ov{bcr^2+u\ell}\ov{c}m}{q})e(\frac{\pm\ov{u\ell}\ov qm}{c})$$
$$S(r\ov u,\pm n;\frac{rcq}{n_1})=S(\ov{c}n_1\ov u,\pm \ov{rc}n_1n;{q})
S(\ov qr\ov u,\pm \ov qn;{rc}/n_1).$$
Therefore, the $(b,u)$ sum in $q^{-1/2}\mathcal{C}(m,n;\frac{rcq}{n_1})$ splits into a product of two sums of respective moduli $rc/n_1$ and $q$. 

The modulus $rc$ sum is denoted by
\begin{equation}\label{Mrcsum}
M_{n_1,r}(m,n,\ell;rc):=\sumstar_{u(c)}e(\frac{\pm\ov{u\ell}\ov qm}{c})S(\ov qr\ov u,\pm \ov qn;{rc/n_1}).	
\end{equation}
The modulus $q$ sum is given by
\begin{multline*}
N_{\ov{cr}}(m,n,\ell;q):=\frac{1}{q}\sumsumstar_\stacksum{b(q),u(q)}{(bcr^2+u\ell,q)=1}\what K(b)e(\frac{\pm\ov{bcr^2+u\ell}\ov{c}m}{q})S(\ov{c}n_1\ov u,\pm \ov{rc}n_1n;{q})\\=\frac{1}{q^{1/2}}\sumsumstar_\stacksum{b(q),u(q)}{(b+u\ell,q)=1}\what K(b)e(\frac{\pm\ov c^2\ov r^2\ov{b+u\ell}m}{q})\Kl_2(\pm \ov{c}^3\ov{r}^3 n_1^2n\ov u;{q})=	\sumstar_{u(q)}L_{\pm\ov c^2\ov r^2m,\ell}(u;q)\Kl_2(\pm \ov{c}^3\ov{r}^3 n_1^2n\ov u;{q})
\end{multline*}
where
\begin{equation}\label{definition-L}
L_{\alpha,\beta}(u;q):=\frac{1}{q^{1/2}}\sum_\stacksum{b(q)}{(b+\beta u,q)=1}\what K(b)e\left(\frac{\alpha\, \ov{b+\beta u}}{q}\right).
\end{equation}
We will sometime suppress the parameters $\alpha$ and $\beta$, and abbreviate $L_{\alpha,\beta}(u;q)$ as $L(u;q)$.

From these notations we find that the sum $\mathrm{Main}_{00}$ in \eqref{n1qcoprime} becomes
\begin{multline}
\mathrm{Main}_{00}=\frac{X}{L^\star Cq^{1/2}}\sum_{\pm\pm}\sum_\stacksum{c\leq  2C}{(c,q)=1}\frac{1}{c}\sum_\stacksum{\ell\in\rmL}{(\ell,c)=1}\ov{\lambda(1,\ell)}\ell^{-iv}\sum_\stacksum{n_1|rc}{(n_1,q)=1}\sum_{m,n}\ov{\lf(m)}\frac{\lambda(n,n_1)}{nn_1}\times\\
M_{n_1,r}(m,n,\ell;rc)N_{\ov{cr}}(m,n,\ell;q) \mcW_{\pm\pm}(\frac{m}{c^2q^2/X},\frac{n_1^2n}{c^3q^3r/XL}).\label{after-voronoi}
\end{multline}

We break the $c$-sum into $O(\log q)$ dyadic intervals and for $C'\leq 2C$ we evaluate the truncated version of  $ S_V(K)$ where $c\sim C'$. Here $C'$ satisfies
$$X^{1/2-\eta}/q\leq C'\leq 2C=(XL/q)^{1/2}.$$
We set
\begin{equation}\label{MNdef}
M=Z^2\frac{{C'}^2q^{2+2\varepsilon}}{X},\ N=Z^3\frac{{C'}^3q^{3+3\varepsilon}r}{XL}.	
\end{equation}
By Lemma \ref{lem-vlocalize} we have
$$m\asymp M, nn^2_1\asymp N$$
which we abbreviate by $$m\approx M,\ nn_1^2\approx N.$$
\begin{remark}\label{nottoosmall}
From this discussion we see that $C'$ cannot be too small: we have $m/M \leq 1$ and since $m\geq 1$, then
$$C'\geq X^{1/2}/Zq^{1+\varepsilon},$$ and this also implies that $N$ is not too small
$$N\geq \frac{X^{1/2}r}L.$$
\end{remark}

\subsection{Cauchy--Schwarz}\label{CSsec}
We will now apply Cauchy--Schwarz inequality with the $n,n_1$ variables outside (to get rid of the factor $\lambda(n,n_1)$) but we need some preparation.

 We factor $c=c_1c_2$ with $$c_1\leq C',\ {n_1|rc_1},\ c_1|(n_1r)^\infty\hbox{ and }(c_2,n_1r)=1.$$

Then we apply Cauchy--Schwarz and \eqref{eq-RS2} to remove the $\GL_3$ coefficients in \eqref{after-voronoi}. Using Lemma \ref{lem-vlocalize} the sum $\mathrm{Main}_{00}$ is bounded by four terms (for the various choices of $\pm,\pm$) of the shape

$$\label{after-cauchy}
\frac{q^{o(1)}X}{LCq^{1/2}}\frac{1}{C'N}A^{1/2}B^{1/2}$$
with
$$A=\sumsum_{nn_1^2\approx N}{|\lambda(n,n_1)|^2n_1}\ll q^{o(1)}N=q^{o(1)}Z^3{C'}^3q^{3}r/XL$$
\begin{multline*}
B=\sumsum_\stacksum{c_1,nn_1^2\approx N}{(n_1,q)=1}n_1\biggl| \sum_\stacksum{\ell\in\rmL}{(\ell,c_1)=1}\ov{\lambda(1,\ell)}\ell^{-iv}\sum_{m\leq M}\ov{\lf(m)}\\
\times
\sum_\stacksum{c_2\sim C'/c_1}{(c_2,q\ell)=1}M_{n_1,r}(m,n,\ell;rc_1c_2)N_{\ov {c_1c_2r}}(m,n,\ell;q)\mcW_{\pm\pm}(\frac{m}{c_1^2c_2^2q^2/X},\frac{n_1^2n}{c_1^3c_2^3q^3r/XL}) \biggr|^2 U\left(\frac{n}{N/n^2_1}\right).
\end{multline*}
Here $U$ is a smooth function with compact support contained in $(0,\infty)$ satisfying \eqref{eqtestfct} with $Z_U=1$.

Hence we obtain that
\begin{equation}\label{main-after-cauchy}
\mathrm{Main}_{00}\ll \frac{q^{o(1)}X^{3/2}}{r^{1/2}Z^{3/2}L^{1/2}C{C'}^{5/2}q^2}B^{1/2}.
\end{equation}
       
 After opening the square, the second factor $B$ equals
\begin{gather}\label{eqprepoisson}
B=\sumsum_{c_1,n_1}n_1\sumsum_\stacksum{\ell,\ell'}{m,m'}\sumsum_{c_2,c_2'}\times\\
\sum_{n\geq 1} M_{n_1,r}(m,n,\ell;rc_1c_2)\ov{M_{n_1,r}(m',n,\ell';rc_1c'_2)}N_{\ov {c_1c_2r}}(m,n,\ell;q)\ov{N_{\ov {c_1c'_2r}}(m',n,\ell';q)}\mcW\left(\frac{n}{N/n^2_1}\right),	\nonumber
\end{gather}
 where 
 \begin{equation}\label{weight-before-poisson}
 \mcW\left(\frac{n}{N/n^2_1}\right)=U\left(\frac{n}{N/n^2_1}\right)\mcW_{\pm\pm}(\frac{m}{c^2_1c^2_2q^2/X},\frac{n_1^2n}{c^3_1c^3_2q^3r/XL})\ov{\mcW_{\pm\pm}}(\frac{m'}{c^2_1{c'_2}^2q^2/X},\frac{n_1^2n}{c^3_1{c'_2}^3q^3r/XL}).
 \end{equation}
\begin{remark}\label{remarksize}
At this point it may be useful to recall the typical size of the various quantities involved: we should imagine that $c_1=n_1=r=1$ and
$$X=q^3,\ L=q^\eta\hbox{ with $\eta>0$ as small as need be},\ C=qL^{1/2}$$
$$q^{1/2}\leq C'\leq qL^{1/2}.$$
$$M={C'}^2q^{-1}\in[1,qL],\ N= {C'}^3q^3/q^3L={C'}^3/L\in[q^{3/2}/L,q^3L^{1/2}]$$
and
$$c_2c'_2q={C'^2}q\in[q^2,q^3L^{1/2}],\ c_2c'_2\approx Q/q\in[q,q^2L^{1/2}].$$
In particular
$$N={C'}^3/L\geq (c_2c'_2q)^{1/2}={C'}q^{1/2};$$
therefore we are in the {\em Polya--Vinogradov} range where applying the Poisson summation formula to \eqref{eqprepoisson} is beneficial. 


\end{remark}

We apply Poisson formula to the $n$-variable keeping in mind that 
\begin{equation}\label{eq-MMNN} n\mapsto M_{n_1,r}(m,n,\ell;rc_1c_2)\ov{M_{n_1,r}(m',n,\ell';rc_1c'_2)}N_{\ov {c_1c_2r}}(m,n,\ell;q)\ov{N_{\ov {c_1c'_2r}}(m',n,\ell';q)}
\end{equation}
 is periodic of period $qk:=qrc_1c_2c_2'/n_1$, and see that \eqref{eqprepoisson} equals
\begin{equation}\label{eqpostpoisson}
\sumsum_{c_1,n_1}n_1\sumsum_\stacksum{\ell,\ell'}{m,m'}\sumsum_{c_2,c_2'}\frac{N}{n^2_1(qk)^{1/2}}\sum_{n\in\Zz}\mathrm{FT}(n,m,m',\ell,\ell';qk)\what \mcW(n/N^*),
\end{equation}
 where
       \begin{multline*}
\mathrm{FT}(n,m,m',\ell,\ell';qk)=\sumsum_{u,u'\mods q}
L_{\pm\ov c^2\ov r^2m,\ell}(u;q)\ov{L_{\pm\ov {c'}^2\ov r^2m',\ell'}(u';q)}\times\\
\frac{1}{\sqrt{qk}}\sum_{v\mods {qk}}\Kl_2(\pm \ov{c}^3\ov{r}^3 n_1^2v\ov u;{q})
\Kl_2(\pm \ov{c'}^3\ov{r}^3 n_1^2v\ov u';{q}) M_{n_1,r}(m,v,\ell;rc_1c_2)\ov{M_{n_1,r}(m',v,\ell';rc_1c'_2)}\, e\left(\frac{nv }{qk}\right)
\end{multline*} 
and
     \begin{equation}\label{defN*}
       N^*:=qkn_1^2/ N	\asymp \frac{n_1}{c_1}\frac{qr{C'}^2}{N}(\asymp \frac{n_1XL}{Z^3c_1C'q^{2+3\varepsilon}}).
       \end{equation}
    Here we can truncate the dual $n$-sum in \eqref{eqpostpoisson} at $|n|\ll q^{2\varepsilon}Z N^*$. 
    \begin{remark}The truncation of the dual $n$-sum follows from the fact that $\mcW(n/N^*)$ is negligibly small when $|n|\gg q^{2\varepsilon}Z N^*$, as can be seen from the expression in \eqref{expansion-of-W}. Indeed, from \eqref{weight-before-poisson} the Fourier transform $\what \mcW(w)$ after Poisson summation equals
\begin{equation}
\what \mcW(w)=\int_{\mathbb{R}}U\left(y\right)\mcW_{\pm\pm}(-,\frac{yN}{c^3_1c^3_2q^3r/XL})\ov{\mcW_{\pm\pm}}(-,\frac{yN}{c^3_1{c'_2}^3q^3r/XL})e(-wy)\mathrm{d}y.
\end{equation}
From \eqref{expansion-of-W}, we see that the factor in the integral that depends on $y$ is of the form 
$$\int_{\mathbb{R}}U\left(y\right)e\big(-3(\frac{yNXL}{c^3_1c^3_2q^3r}\xi_1)^{1/3}+3(\frac{yNXL}{c^3_1{c'_2}^3q^3r}\xi_2)^{1/3}-wy\big)\mathrm{d}y$$   
where $\xi_1, \xi_2\asymp 1$,    
       and for $w\neq 0$ by applying integration by parts, this is
       $$\ll \left(\frac{q^{\varepsilon}Z}{|w|}\right)^A$$
       upon noting that $\frac{NXL}{{C'}^3q^3r}\asymp q^{3\varepsilon}Z^3$ by \eqref{MNdef}. Hence $\what \mcW(w)$ is negligibly small when $|w|\gg  q^{2\varepsilon}Z$.
       \end{remark}

\subsection{Computation of $\mathrm{FT}$}\label{computation-fourier-transform}
Recall $k=rc_1c_2c'_2/n_1$. We have $(q,k)=1$ and we split the above sum $\mathrm{FT}(n,m,m',\ell,\ell';qk)$ as a product of sums $\FT(n;q)$ and $\FT(n;k)$ of respective moduli $q$ and $k$ (to simplify notations we do not display the dependency in $m,m',\ell,\ell'$ in these expressions).
  \subsubsection{The $k$-sum} The $k$-sum equals
$$\FT(n;k):=\frac{1}{\sqrt{k}}\sum_{v(k)}M_{n_1,r}(m,v,\ell;rc_1c_2)\ov{M_{n_1,r}(m',v,\ell'  ;rc_1c'_2)}\, e\left(\frac{ nv \bar{q}}k\right).$$
\begin{lemma}\label{bound-k-part}
 We have the following bounds
\begin{equation*}
\FT(0;k)\ll \sqrt{k}rc_1c_2\mathop{\sum_{d|c_1c_2}\sum_{d'|c_1c_2}}_{(\ell'   d,\ell d')|(m\ell'  -m'\ell )} (d ,d'),
\end{equation*}
and
\begin{equation*}
\begin{split}
\FT(n;k)\ll \sqrt{k}\sum_{d_1|c_1}d_1\sum_{d'_1|c_1}d'_1
\mathop{\sumstar_\stacksum{x_1(rc_1/n_1)}
{\ell  n_1x_1\equiv \mp m\mods{d_1}}}\sumsum_\stacksum{d_2|(c_2,\ell n_1c'_2+nm)}{d'_2|(c'_2,\ell'  n_1c_2+nm')}d_2d'_2.
\end{split}\end{equation*}
\end{lemma}
\proof
To see this, we recall
\begin{equation}\label{MnRam}
M_{n_1,r}(m,n,\ell;rc)=\sum_{d|c}d\mu\left(\frac{c}{d}\right)\sumstar_\stacksum{x(rc/n_1)}{\ell n_1x\equiv \mp m\mods d}e\left(\frac{\pm \bar{q}n\bar{x}}{rc/n_1}\right).	
\end{equation}
Then
\begin{equation}\label{mcM(n)}
\begin{split}
\FT(n;k)=&\frac{1}{\sqrt{k}}\sum_{d|c_1c_2}d\mu\left(\frac{c_1c_2}{d}\right)\sum_{d'|c_1c_2'}d'\mu\left(\frac{c_1c_2'}{d'}\right)\times\\ &\ \ \sumstar_\stacksum{x(rc_1c_2/n_1)}{\ell  n_1x\equiv \mp m\mods d}\sumstar_\stacksum{x'(rc_1c'_2/n_1)}{\ell'   n_1x'\equiv \mp m'\mods {d'}}\sum_{v(k)}e\left(\frac{(\pm\bar{q}\bar{x}c'_2\mp \bar{q}\ov{x'}c_2+n\bar{q})v}{k}\right)\\
=&\sqrt{k}\sum_{d|c_1c_2}d\mu\left(\frac{c_1c_2}{d}\right)\sum_{d'|c_1c_2'}d'\mu\left(\frac{c_1c_2'}{d'}\right)\mathop{\sumstar_\stacksum{x(rc_1c_2/n_1)}{\ell  n_1x\equiv \mp m\mods d}\sumstar_\stacksum{x'(rc_1c'_2/n_1)}{\ell'   n_1x'\equiv \mp m'\mods{d'}}}_{\bar{x}c'_2-\ov{x'}c_2\equiv \mp n\mods k}1.
\end{split}
\end{equation}
We calculate the case where $n\equiv 0\mods k$ first. 

For $n\equiv 0\mods k$, the congruence $\bar{x}c'_2-\ov{x'}c_2\equiv \mp n\mods k$ forces 
\begin{equation}\label{cequalc'}
c_2=c'_2 	
 \end{equation}
 and $x\equiv x' \mods{rc_1c_2/n_1}$. Therefore
$$\FT(0;k)=\sqrt{k}\sum_{d|c_1c_2}d\mu\left(\frac{c_1c_2}{d}\right)\sum_{d'|c_1c_2}d'\mu\left(\frac{c_1c_2}{d'}\right)\mathop{\sumstar_\stacksum{x(rc_1c_2/n_1)}{\ell  n_1x\equiv \mp m\mods{d}}}_{\ell'   n_1x\equiv \mp m'\mods{d'}}1.$$
Notice that since we have $(\ell,c)=1$ (see \eqref{S'sum}), for $d|c$ we have $(d,\ell)=1$. The system of equations $\ell  n_1x\equiv \mp m\mods{d}$ and $\ell'   n_1x\equiv \mp m'\mods{d'}$ has a unique solution modulo $[\frac{d}{(n_1,d)},\frac{d'}{(n_1,d')}]$; moreover it implies that $(\ell'   d,\ell d')|(m\ell'  -m'\ell )$. The number of solutions for $x\mods {\frac{rc_1c_2}{n_1}}$ is therefore given by
$$\frac{rc_1c_2/n_1}{[\frac{d}{(n_1,d)},\frac{d'}{(n_1,d')}]}=\frac{rc_1c_2}{[n_1,d,d']}.$$
Hence $$\FT(0;k)= \sqrt{k}\mathop{\sum_{d|c_1c_2}\sum_{d'|c_1c_2}}_{(\ell'   d,\ell d')|(m\ell'  -m'\ell )} d\mu\left(\frac{c_1c_2}{d}\right)d'\mu\left(\frac{c_1c_2}{d'}\right)\frac{rc_1c_2}{[n_1,d,d']}.$$
In particular,
\begin{equation}
\begin{split}
\FT(0;k)\ll& \sqrt{k}rc_1c_2\mathop{\sum_{d|c_1c_2}\sum_{d'|c_1c_2}}_{(\ell'   d,\ell d')|(m\ell'  -m'\ell )}(d,d').
\end{split}\end{equation}

Next we consider the case where $n\neq 0$ in $\FT(n;k)$. In \eqref{mcM(n)} we write $d=d_1d_2$ where $d_1|c_1$, $d_2|c_2$ and $d'=d'_1d'_2$ where $d'_1|c_1$, $d'_2|c'_2$. Then the sum splits as 
\begin{equation}
\begin{split}
\FT(n;k)=&\sqrt{k}\sum_{d_1|c_1}d_1\mu\left(\frac{c_1}{d_1}\right)\sum_{d'_1|c_1}d'_1\mu\left(\frac{c_1}{d'_1}\right)
\sum_{d_2|c_2}d_2\mu\left(\frac{c_2}{d_2}\right)\sum_{d'_2|c'_2}d'_2\mu\left(\frac{c'_2}{d'_2}\right)\\
&\mathop{\sumstar_\stacksum{x_1(rc_1/n_1)}{\ell  n_1x_1\equiv \mp m\mods{d_1}}\sumstar_\stacksum{x'_1(rc_1/n_1)}{\ell'   n_1x'_1\equiv \mp m'\mods{d'_1}}}_{\ov{x_1}c'_2-\ov{x'_1}c_2\equiv \mp n\mods{rc_1/n_1}} \,\, \mathop{\sumstar_\stacksum{x_2(c_2)}{\ell  n_1x_2\equiv \mp m\mods{d_2}}\sumstar_\stacksum{x'_2(c'_2)}{\ell'   n_1x'_2\equiv \mp m'\mods{d'_2}}}_{\ov{x_2}c'_2-\ov{x'_2}c_2\equiv \mp n\mods{c_2c'_2}}1\\
:=& \sqrt{k}\FT_1(n;k)\FT_2(n;k),
\end{split}
\end{equation}
where
$$\FT_1(n;k)=\sum_{d_1|c_1}d_1\mu\left(\frac{c_1}{d_1}\right)\sum_{d'_1|c_1}d'_1\mu\left(\frac{c_1}{d'_1}\right)\mathop{\sumstar_\stacksum{x_1(rc_1/n_1)}{\ell  n_1x_1\equiv \mp m\mods{d_1}}\sumstar_\stacksum{x'_1(rc_1/n_1)}{\ell'   n_1x'_1\equiv \mp m'\mods{d'_1}}}_{\ov{x_1}c'_2-\ov{x'_1}c_2\equiv \mp n\mods{rc_1/n_1}}1,$$
and
$$\FT_2(n;k)=\sum_{d_2|c_2}d_1\mu\left(\frac{c_2}{d_2}\right)\sum_{d'_2|c'_2}d'_2\mu\left(\frac{c'_2}{d'_2}\right)\mathop{\sumstar_\stacksum{x_2(c_2)}{\ell  n_1x_2\equiv \mp m\mods{d_2}}\sumstar_\stacksum{x'_2(c'_2)}{\ell'   n_1x'_2\equiv \mp m'\mods{d'_2}}}_{\ov{x_2}c'_2-\ov{x'_2}c_2\equiv \mp n\mods{c_2c'_2}}1.$$
In $\FT_1(n;k)$, the term $x'_1\mods{rc_1/n_1}$ is completely determined by $x_1\mods{rc_1/n_1}$. Therefore estimating trivially, one sees that
$$\FT_1(n;k)\ll \sum_{d_1|c_1}d_1\sum_{d'_1|c_1}d'_1
\mathop{\sumstar_\stacksum{x_1(rc_1/n_1)}
{\ell  n_1x\equiv \mp m\mods{d_1}}1}.$$

For $\FT_2(n;k)$, the congruence conditions there imply that $d_2|\ell n_1c'_2+nm$ and $d'_2|\ell'  n_1c_2+nm'$. Therefore
$$\FT_2(n;k)\ll \sumsum_\stacksum{d_2|(c_2,\ell n_1c'_2+nm)}{d'_2|(c'_2,\ell'  n_1c_2+nm')}d_2d'_2.$$
\qed

\subsubsection{The $q$-sum}    \label{secq-sum}   
       The $q$-sum equals
\begin{multline}\label{eq-qsum2}
\FT(n;q)=\frac{1}{\sqrt q}\sumsum_{u,u'\mods q}
L_{\pm\ov c^2\ov r^2m,\ell}(u;q)\ov{L_{\pm\ov {c'}^2\ov r^2m',\ell'}(u';q)}\times\\
\sum_{v\mods q}\Kl_2(\pm \ov{c}^3\ov{r}^3 n_1^2v\ov u;{q})\Kl_2(\pm \ov{c'}^3\ov{r}^3 n_1^2v\ov u';{q})\, e\left(\frac{\ov kv n}q\right).
\end{multline}
It will be useful to transform it to make it amenable to a sheaf-theoretic treatment.

For $\alpha,\beta,\gamma,\alpha',\beta',\gamma'\in\Fqt$ we recall (see \eqref{definition-L})
\begin{equation}\label{L-sum2}
L(u;q)=L_{\alpha,\beta}(u;q):=\frac{1}{\sqrt{q}}\sum_{b(q)}\widehat{K}(b)e\left(\frac{\alpha\overline{(b+\beta u)}}{q}\right)=\frac{1}{\sqrt{q}}\sum_{a\in\Fq}K(a)\Kl_2(\alpha a;q)e(-\frac{\beta au}q),
\end{equation}
the latter identity following from the expression of the Fourier transform \eqref{fourierdef}.
We also define
$$M(u):=\Kl_2(\gamma u;q)$$
and define likewise $L'(u;q),\ M'(u')$ with $\alpha',\beta',\gamma'$.
The following choices of values of the parameters correspond to our initial problem:
\begin{gather}\alpha=\pm \ov c^2\ov r^2m,\ \beta=\ell,\ \alpha'=\pm \ov {c'}^2\ov r^2m',\ \beta'=\ell'\nonumber\\
\gamma=\pm \ov{c}^3\ov{r}^3 n_1^2,\ \gamma'=\pm \ov{c'}^3\ov{r}^3 n_1^2,\ \delta=\ov k n,\label{eq-actual}
\end{gather}
we see (by switching the $u,u'$ sums and the $v$ sum) that \eqref{eq-qsum2} equals
$${\sqrt q}\sum_{v\mods q}M\star L(v)\ov{M'\star L'(v)}e(\frac{\delta v}q)$$
where $M\star L$ denotes the normalized multiplicative convolution 
$$M\star L(v)=\frac{1}{q^{1/2}}\sum_{u\in\Fqt }M(u)L(v/u)=\frac{1}{q^{1/2}}\sum_{u\in\Fqt }M(v u)L(1/u).$$

To evaluate such sums further we will need to make a few elementary transformations. 

By Plancherel formula we have
\begin{equation}\label{MLsum}
\sum_{v}M\star L(v)\ov{M'\star L'(v)}e(\frac{\delta v}q)=\sum_{v}\what{M\star L}(v)\ov{\what{M'\star L'}(v-\delta)}=\sum_{v}Z(v)\ov{Z'(v-\delta)}\end{equation}
say. Let us now compute $Z(v)$:
$$Z(v)=\frac{1}{q^{1/2}}\frac{1}{q^{1/2}}\sum_{x\in\Fqt }\sum_{u}M(x)L(u/x)e(\frac{vu}{q})=\frac{1}{q^{1/2}}\sum_{x\in\Fqt}M(x)\what L(xv).$$
By \eqref{L-sum} we have 
$$\what{L}(x)=K(x/\beta)\Kl_2(\alpha/\beta x;q)$$
and
\begin{equation}\label{Zcompute}
Z(v)=\frac{1}{q^{1/2}}\sum_{x\in\Fqt}\Kl_2(\beta\gamma x;q)K(xv)\Kl_2(\alpha xv;q).	
\end{equation}

In \S \ref{sec-sqroot} we will prove the following
\begin{proposition}\label{sqrootcancel} 
Let $T_\mcF(\Fq)$ be the subgroup of $\Fqt$ defined by
$$ T_\mcF(\Fq)=\{\lambda\in\Fqt,\ [\times\lambda]^*\mcF\hbox{ is geometrically isomorphic to }\mcF\}.
$$
 Assuming that the sheaf $\mcF$ is good, then for any $\alpha,\beta,\alpha',\beta',\gamma,\gamma',\delta\in\Fqt$, we have
$$\sum_{v}Z(v)\ov{Z'(v-\delta)}=O(q^{1/2}).$$
If $\delta=0$ the above bound holds unless
$$\alpha/\alpha'=\beta\gamma/\beta'\gamma'\in T_\mcF(\Fq)$$
in which case
$$\sum_{v}Z(v)\ov{Z'(v)}=c_\mcF(\alpha/\alpha')q+O(q^{1/2})$$
for $c_\mcF(\alpha/\alpha')$ some complex number of modulus $1$.
Here the implicit constants depend only on $C(\mcF)$.
\end{proposition}

Returning to our original sum, we see that \eqref{eq-qsum2} is  $O(q)$ unless (observe that the $r$ variable is gone)
$${c'}^2m/c^2m'={c'}^3\ell/c^3\ell'\in T_\mcF(\Fq)$$ in which case 
\eqref{eq-qsum2} equals $C({c'}^2m/c^2m')q^{3/2}+O(q)$ with $|C({c'}^2m/c^2m')|=1.$

\section{Contribution of the $n=0$ frequency}\label{contribution-of-zero}

In this section we bound the contribution  to \eqref{eqpostpoisson} of the frequency $n=0$. By \eqref{cequalc'}, we then have
\begin{equation}\label{cequal}
c_2=c_2',\ c=c',\ k=rc_1c_2^2/n_1.	
\end{equation}

We use the case $\delta=0$ of Proposition \ref{sqrootcancel}: by \eqref{cequal} and \eqref{eq-actual} we have that \eqref{eq-qsum2} is  $O(q)$ unless we have the congruence modulo $q$
$$m/m'=\ell/\ell'\in T_\mcF(\Fq)$$ in which case 
\eqref{eq-qsum2} equals $C(m/m')q^{3/2}+O(q)$ with $|C(m/m')|=1.$
 
 The contribution of the $n=0$ frequency to \eqref{eqpostpoisson} is bounded by
\begin{multline*}\ll 
\frac{N}{q^{1/2}}\sumsum_{n_1,c=c_1c_2}\frac{1}{n_1k^{1/2}}\sumsum_{\ell,\ell'}|\lambda(1,\ell)||\lambda(1,\ell')|\\\times\sumsum_{m,m'\leq M}|\lf(m)||\lf(m')|(q^{3/2}\delta_{m/m'=\ell/\ell'\in T_\mcF(\Fq)}+q)|\FT(0;k)|	
\end{multline*}
\begin{multline*}\ll 
\frac{N}{q^{1/2}}\sumsum_{n_1,c=c_1c_2}\frac{1}{n_1k^{1/2}}\sumsum_{\ell,\ell'}|\lambda(1,\ell')|^2\\\times\sumsum_{m,m'\leq M}|\lf(m)|^2(q^{3/2}\delta_{m/m'=\ell/\ell'\in T_\mcF(\Fq)}+q)|\FT(0;k)|	
\end{multline*}
\begin{multline*}\ll q^{o(1)}
\frac{rN{C'}}{q^{1/2}}\sumsum_{n_1,c}\frac{1}{n_1}\sumsum_{\ell,\ell'}|\lambda(1,\ell')|^2\\\times\sumsum_{m,m'\leq M}|\lf(m)|^2(q^{3/2}\delta_{m\ell'=\ell m'\mods q}+q)\sumsum_\stacksum{d,d'|c}{(\ell' d, \ell d')|(m\ell'  -m'\ell )} ({d},{d'}).	
\end{multline*}
Writing $b=(d,d')$ the sum is bounded by
$$\ll q^{o(1)}
\frac{rN{C'}}{q^{1/2}}\sumsum_{n_1,b|c\sim C'}\frac{b}{n_1}\sumsum_\stacksum{\ell,\ell'\sim L,m,m'\leq M}{b|\ell'm-\ell m'}|\lambda(1,\ell')|^2 |\lf(m)|^2(q^{3/2}\delta_{m\ell'=\ell m'\mods q}+q)
$$
$$\ll q^{o(1)}
\frac{rN{C'}}{q^{1/2}}LM\sum_{b|c\sim C'}b\left(q^{3/2}(\frac{LM}{qb}+1)+q(\frac{LM}{b}+1)\right)$$
$$\ll q^{o(1)}
\frac{rN{C'}}{q^{1/2}}LMC'(q^{1/2}{LM}+C'q^{3/2}+q{LM}+{C'}q)
\ll q^{o(1)}
\frac{rN{C'}^2LM}{q^{1/2}}(C'q^{3/2}+q{LM}).$$

Taking the squareroot of this term and multiplying by $\frac{X^{3/2}}{r^{1/2}Z^{3/2}L^{1/2}C{C'}^{5/2}q^2}$, we see that the contribution of these terms to \eqref{main-after-cauchy} and then to \eqref{n1qcoprime} is bounded by
\begin{equation}\label{n=0bound}
\begin{split}
&q^{o(1)}\frac{r^{1/2}XM^{1/2}}{CL^{1/2}q^{3/4}}\left({C'}^{1/2}q^{3/4}+q^{1/2}L^{1/2}M^{1/2}\right)\\
\ll&q^{o(1)}r^{1/2}\left(Z\frac{X^{3/4}q^{3/4}}{L^{1/4}}+Z^2 L^{1/2}X^{1/2}q^{5/4}\right).	
\end{split}\end{equation}

In particular for $X=q^3$ we obtain
$$q^{o(1)}r^{1/2}(Z\frac{q^{3}}{L^{1/4}}+Z^2L^{1/2}q^{3-1/4})$$
which when $r=1$ is non-trivial for $1<L<q^{1/2}$.

\section{Contribution from the $n\not=0$ frequencies} \label{secn=0}

Recall from \eqref{eqpostpoisson} that
\begin{equation}
\begin{split}
B_{n\neq 0}=&\frac{N}{q^{1/2}}\sumsum_{c_1,n_1}\frac1{n_1}\sum_{\ell\in\rmL}\ov{\lambda(1,\ell)}{\ell}^{-iv}\sum_{\ell'\in\rmL}\lambda(1,\ell'){\ell'}^{iv}\sum_{m\leq M}\ov{\lf(m)}\sum_{m'\leq M}{\lf(m')}\sum_{c_2\sim C'/c_1}\sum_{c_2'\sim C'/c_1}\\
&\frac{1}{k^{1/2}}\sum_{n\not=0}\mathrm{FT}(n;q)\mathrm{FT}(n;k)\what \mcW(n/N^*).
\end{split}\end{equation}

We consider two cases: $n\not=0\mods q$ and $n\equiv0\mods q$. 
\subsection{$n\neq 0\mods q$}
By Proposition \ref{sqrootcancel}, we have 
$$\mathrm{FT}(n;q)\ll q.$$
 Combining this with Lemma \ref{bound-k-part}, gives
\begin{equation*}
\begin{split}
B_{n\neq 0 \mods q}\ll &{q^{o(1)}Nq^{1/2}}\sumsum_{c_1,n_1}\frac{1}{n_1}\sum_{\ell\in\rmL}|\ov{\lambda(1,\ell)}|\sum_{\ell'\in\rmL}|\lambda(1,\ell')|\sum_{m}|\ov{\lf(m)}|\sum_{m'}|{\lf(m')}|\sum_{c_2}\sum_{c_2'}\\
&\sum_\stacksum{n\ll ZN^*}{n\neq 0 \mods q}\sum_{d_1,d'_1|c_1}d_1d'_1
\mathop{\sumstar_\stacksum{x_1(rc_1/n_1)}
{\ell  n_1x_1\equiv \mp m\mods{d_1}}}\sumsum_\stacksum{d_2|(c_2,\ell n_1c'_2+nm)}{d'_2|(c'_2,\ell'  n_1c_2+nm')}d_2d'_2|\what \mcW(n/N^*)|.
\end{split}\end{equation*}
By using $|\lambda(1,\ell)||\lambda(1,\ell')||\lf(m)||\lf(m')|\ll |
\lambda(1,\ell')\lf(m)|^2+|\lambda(1,\ell)\lf(m')|^2$, we have
\begin{equation*}
\begin{split}
B_{n\neq 0 \mods q}\ll &q^{o(1)}Nq^{1/2}\sumsum_{c_1,n_1}\frac{c_1}{n_1}\sum_{\ell\in\rmL}\sum_{\ell'\in\rmL}|\lambda(1,\ell)|^2 \sum_\stacksum{n\ll ZN^*}{n\neq 0 \mods q}\sum_{d_1|c_1}d_1
\\
&\sum_{m\leq M}\sum_{m'\leq M}|{\lf(m')}|^2\sum_{c_2\sim C'/c_1}\sum_{c_2'\sim C'/c_1}\mathop{\sumstar_\stacksum{x_1(rc_1/n_1)}
{\ell  n_1x_1\equiv \mp m\mods{d_1}}}\sumsum_\stacksum{d_2|(c_2,\ell n_1c'_2+nm)}{d'_2|(c'_2,\ell'  n_1c_2+nm')}d_2d'_2.
\end{split}\end{equation*}
The right hand side, when replacing $c_2$ by $c_2d_2$ and $c'_2$ by $c'_2d'_2$, is
\begin{equation*}
\begin{split}
&q^{o(1)}Nq^{1/2}\sumsum_{c_1,n_1}\frac{c_1}{n_1} \sum_\stacksum{n\ll ZN^*}{n\neq 0 \mods q}\sum_{d_1|c_1}d_1\sum_{d_2\ll C'/c_1}d_2\sum_{d'_2\ll C'/c_1}d'_2\sum_{c_2\sim C'/c_1d_2}
\\
&\sum_{\ell'\in\rmL}\sum_\stacksum{m'\leq M}{\ell'  n_1c_2d_2+nm'\equiv 0\mods{d'_2}}|{\lf(m')}|^2\times\sum_{c_2'\sim C'/c_1d'_2}\sum_{\ell\in\rmL}|\ov{\lambda(1,\ell)}|^2\sumstar_{x_1(rc_1/n_1)}
\mathop{\sum_\stacksum{m\leq M}{\ell n_1c'_2d'_2+nm\equiv 0\mods{d_2}}}_{\ell  n_1x_1\equiv \mp m\mods{d_1}}1.
\end{split}\end{equation*}
Bounding the number of solutions in $m$ of the congruence equations  $\ell n_1c'_2d'_2+nm\equiv 0\mods{d_2}$ and $\ell  n_1x_1\equiv \mp m\mods{d_1}$ we see that
\begin{eqnarray*}
\sum_{c_2'\sim C'/c_1d'_2}\sum_{\ell\in\rmL}|\ov{\lambda(1,\ell)}|^2\sumstar_{x_1(rc_1/n_1)}
\mathop{\sum_\stacksum{m\leq M}{\ell n_1c'_2d'_2+nm\equiv 0\mods{d_2}}}_{\ell  n_1x_1\equiv \mp m\mods{d_1}}1&&\ll q^{o(1)}(d_2,n)\frac{rc_1}{n_1}\frac{C'L}{c_1d'_2}\left(1+\frac{M}{d_1d_2}\right)\\
&&\ll	q^{o(1)}(d_2,n)\frac{r}{n_1}\frac{C'L}{d'_2}\left(1+\frac{M}{d_1d_2}\right).
\end{eqnarray*}
Averaging this bound over the remaining variables we obtain
\begin{multline*}
B_{n\neq 0 \mods q}\ll q^{o(1)}rNq^{1/2}{C'}L\sumsum_{c_1,n_1}\frac{c_1}{n^2_1} \sum_\stacksum{n\ll ZN^*}{n\neq 0 \mods q}\\
\sum_{d_2\ll C'/c_1}(d_2,n)\left(c_1d_2+M\right)
\sum_{d'_2\ll C'/c_1}\sum_{c_2\sim C'/c_1d_2}\sum_{\ell'\in\rmL}\sum_\stacksum{m'\leq M}{\ell'  n_1c_2d_2+nm'\equiv 0\mods{d'_2}}|{\lf(m')}|^2.	
\end{multline*}
-- If $\ell'  n_1c_2d_2+nm'\not=0$ we interpret the congruence $\ell'  n_1c_2d_2+nm'\mods{d'_2}$ as $d'_2$ being a divisor of that integer so that the contribution of such terms is bounded by
$$
\ll q^{o(1)}rNq^{1/2}{C'}L\sumsum_{c_1,n_1}\frac{c_1}{n^2_1} ZN^*\sum_{d_2\leq C'/c_1}
\left(c_1d_2+M\right)\frac{C'}{c_1d_2}
LM$$
$$
=q^{o(1)}rZNq^{1/2}{C'}^2L^2M\sumsum_{c_1,n_1}\frac{c_1}{n^2_1} \frac{n_1}{c_1}\frac{qr{C'}^2}{N}\sum_{d_2\leq C'/c_1}
\left(1+\frac{M}{c_1d_2}\right)
$$
$$
=q^{o(1)}r^2Z{C'}^4q^{3/2}L^2M\sumsum_{c_1,n_1}\frac{1}{n_1}\sum_{d_2\leq C'/c_1}
\left(1+\frac{M}{c_1d_2}\right)
$$
$$=q^{o(1)}r^2Z{C'}^5q^{3/2}L^2M
\left(1+\frac{M}{{C'}}\right).$$

-- On the other hand, the contribution of the terms satisfying $\ell'  n_1c_2d_2+nm'=0$ is bounded by (this follows from bounding the number of representations of $m'n'$ as a product of four factors, where $n'=n/(d_2,n)$ and using the Rankin--Selberg bound \eqref{eq:RS})
\begin{multline*}
q^{o(1)}rNq^{1/2}{C'}^3L\sumsum_{c_1,n_1}\frac{1}{n^2_1} \sum_\stacksum{n\ll ZN^*}{n\neq 0 \mods q}\\
\sum_{d_2\ll C'/c_1}(d_2,n)\left(1+\frac{M}{{C'}}\right)
\sum_{c_2\sim C'/c_1d_2}\sum_{\ell'\in\rmL}\sum_\stacksum{m'\leq M}{\ell'  n_1c_2d_2=-nm'}|{\lf(m')}|^2	
\end{multline*}
$$\ll q^{o(1)}rNq^{1/2}{C'}^3L\left(1+\frac{M}{{C'}}\right)\sumsum_{c_1,n_1}\frac{1}{n^2_1}ZN^*M=q^{o(1)}r^2Z{C'}^5q^{3/2}LM\left(1+\frac{M}{{C'}}\right),$$
which is smaller than the previous term.

\subsection{${q|n,\ n\not=0}$} In that case we write $n=qn'$ with $n'\leq q^{o(1)}N^*/q$ and have
$$\FT(n'q;q)\ll q^{3/2}.$$
We have there lost a factor $q^{1/2}$ by comparison with the previous section. On the other hand we can run exactly the same argument as above with $N^*$ replaced by $N^*/q$ and all in all we find that 
 $$B_{q|n,n\not=0}\ll q^{o(1)}r^2Z{C'}^5q^{3/2}L^2M
\left(1+\frac{M}{{C'}}\right)\frac1{q^{1/2}}.$$ 
The non-zero frequencies contribution to \eqref{main-after-cauchy} and hence to \eqref{n1qcoprime} is bounded by 
\begin{gather}\nonumber
q^{o(1)} \frac{X^{3/2}}{r^{1/2}Z^{3/2}L^{1/2}C{C'}^{5/2}q^2} \left(B_{n\neq 0 \mods q}+B_{q|n,n\not=0}\right)^{1/2}\\
\leq q^{o(1)} \frac{X^{3/2}}{r^{1/2}Z^{3/2}L^{1/2}C{C'}^{5/2}q^2}  \left(r^2Z{C'}^5q^{3/2}L^2M\left(1+\frac{M}{{C'}}\right)\right)^{1/2}\nonumber\\
\leq q^{o(1)}r^{1/2}\left(\frac{L^{1/2}X}{q^{1/4}}+ ZL^{3/4}X^{3/4}q^{1/2}\right).\label{nonzerocomb}
\end{gather}

\subsection{Bounding $\mathrm{Main}_0$: Conclusion}\label{q-divide-n1}
Let us recall that the sum $\mathrm{Main}_0$ in \eqref{eqaftervoronoi} was split into two subsums depending on whether $(n_1,q)=1$ or not. 

By \eqref{n=0bound} and \eqref{nonzerocomb} the first subsum \eqref{n1qcoprime} is bounded by
\begin{equation}\label{n1qcoprimefinal}
\ll q^{o(1)}r^{1/2}\Big(Z\frac{X^{3/4}q^{3/4}}{L^{1/4}}+Z^2L^{1/2}X^{1/2}q^{5/4}
+\frac{L^{1/2}X}{q^{1/4}}+ZL^{3/4}X^{3/4}q^{1/2}\Big).	
\end{equation}
The complement sum (when $q|n_1$) is given (rewriting $n_1$ into $qn_1$) by
\begin{multline}
\mathrm{Err}_4=\frac{X}{L^\star Cq^2}\sum_{\pm\pm}\sum_\stacksum{c\leq  2C}{(c,q)=1}\frac{1}{c}\sum_{\ell\in\rmL}\ov{\lambda(1,\ell)}\ell^{-iv}\sum_{n_1|rc}\sum_{m,n}\ov{\lf(m)}\\
\times\frac{\lambda(n,n_1q)}{nn_1}\mathcal{C}(m,n;\frac{rc}{n_1})\mcW_{\pm\pm}\left(\frac{m}{c^2q^2/X},\frac{n_1^2n}{c^3qr/XL}\right)	\label{qdivn1}
\end{multline}
where
\begin{gather*}
\mathcal{C}(m,n;\frac{rc}{n_1})=\frac{1}{q^{1/2}}\sumsumstar_{b(q),u(cq)}\what K(b)e\left(\frac{\pm\ov{bcr^2+u\ell}m}{cq}\right)S\left(r\ov u,\pm n;\frac{rc}{n_1}\right)\\
=\frac{1}{q^{1/2}}\sumsumstar_{b(q),u(q)}\what K(b)e\left(\frac{\pm\ov{bcr^2+u\ell}m\ov c}{q}\right)\times \sumstar_{u\mods c}	
e\left(\frac{\pm\ov{u\ell}m\ov q}{c}\right)S\left(r\ov u,\pm n;\frac{rc}{n_1}\right)\\
=-K(0)\sumstar_{u\mods c}	
e\left(\frac{\pm\ov{u\ell}m\ov q}{c}\right)S\left(r\ov u,\pm n;\frac{rc}{n_1}\right).
\end{gather*}
This last $u$-sum is very similar to the sum \eqref{Mrcsum} discussed previously (the Kloosterman sum $S\left(\ov qr\ov u,\pm \ov q n; rc/n_1\right)$ has just been replaced by 
$S\left(r\ov u,\pm n; rc/n_1\right)$) and is bounded by $(rc)^{1+o(1)}$ (cf. \eqref{MnRam}). Using this bound we obtain that \eqref{qdivn1} is bounded by
\begin{equation}\label{qdivn1final}
\ll q^{o(1)+\theta_3}\frac{X}{LCq^2}L\frac{Z^2 C^2q^2}XrC= q^{o(1)+\theta_3}rZ^2\frac{XL}q.	
\end{equation}
Combining this bound with \eqref{n1qcoprimefinal} we obtain that

\begin{multline}
\mathrm{Main}_0\ll q^{o(1)+\theta_3}rZ^2\frac{XL}q+	q^{o(1)}r^{1/2}\Big(Z\frac{X^{3/4}q^{3/4}}{L^{1/4}}\\
+Z^2L^{1/2}X^{1/2}q^{5/4}+\frac{L^{1/2}X}{q^{1/4}}+ZL^{3/4}X^{3/4}q^{1/2}\Big). \label{final-for-main}	
\end{multline}

\section{Bounding $S_V(K,X)$: the final steps}

\subsection{Bounding $\mathrm{Err}_1$, $\mathrm{Err}_2$ and $\mathrm{Err}_3$ }\label{Err23}
Our treatment of the error terms $\mathrm{Err}_1$, $\mathrm{Err}_2$ and $\mathrm{Err}_2$  is similar to that of $\mathrm{Main}_0$ and in the end these terms  yield smaller contributions. We will again start by using  Voronoi summation on the $m$ and the $n$ variables but, from this point on, the argument becomes much simpler than that of \S \ref{CSsec}. For instance for $\mathrm{Err}_1$ it is sufficient to bound the resulting sums trivially (no need to apply Cauchy--Schwarz to smooth the $n$ variable). Since these arguments have already been presented in full details and in a more complex form in \S \ref{CSsec}, we will be brief and pass over various technical points (for instance we will often assume the coprimality of different variables).

\subsubsection{$\mathrm{Err}_1$}We refer to \eqref{eqErr2} for the shape of this term. As before we apply Voronoi summation formulas to the $m$ and $n$ sums. The chief difference is that the exponential  $$n\mapsto e(\frac{an}c)$$ has modulus $c\leq C$ instead of $cq$. The dual sum has length $$(cZ)^3/XL\leq Z^3X^{3/2}L^{3/2}/(Xq^{3/2}L)=Z^3(XL)^{1/2}/q^{3/2}$$
which is rather small ($\leq L^{1/2}$ for $X=q^3$) and the exponential is (essentially) transformed into Kloosterman sums typically of the shape 
\begin{equation}\label{dualn}
	n\mapsto \Kl_2(r\ov{a}n;rc).
\end{equation}

For the $m$-sum, the  function $$m\mapsto K(mr^2)e(\frac{-am\ell}c)$$ has period $cq$, therefore the dual $m$-sum (after Voronoi) has length
$$(cqZ)^2/X\leq Z^2qL$$ and is weighted
 (essentially) by the function
\begin{equation}\label{dualm}m\mapsto \widecheck K(\pm\ov{c}^2\ov{r}^2m)e(\pm \frac{\ov{a\ell q^2}m}{c})	
\end{equation}
where
$$\widecheck K(m)=\frac{1}{\sqrt q}\sum_{b\in\Fqt}\what K(b)e(\frac{\ov b m}q)=\bigl(\what K\star e(\frac{\cdot}{q})\bigr)(m).$$
The convolution $\widecheck K$ is a trace function unless $\mcF$ is geometrically isomorphic to $[x\mapsto \alpha x]^*\KL_2,\ \alpha\in\Fqt$ (whose Fourier transform is $[x\mapsto x^{-1}]^*\mcL_{\psi}$ for $\mcL_{\psi}$ the Artin-Schreier sheaf attached to an additive character depending on $\alpha$), but such situation is excluded by both the MO and the SL assumptions. Summing the product of \eqref{dualn} and \eqref{dualm} over $a\mods c$ the resulting sum is bounded by
$$\ll \|\widecheck K\| c^{1/2+o(1)}$$ and therefore
we obtain
\begin{equation}
	\begin{split}\label{Err1bound}
\mathrm{Err}_1&\ll \frac{q^{o(1)}}{LCq}L\sum_{c\leq C}\frac{1}c\frac{XL}{(cZ)^{3/2}}\frac{(cZ)^3}{XL}\frac{X}{cqZ}\frac{(cqZ)^2}{X}c^{1/2}\\&=
 \frac{q^{o(1)}Z^{5/2}}{Cq}\sum_{c\leq C}{c^{2}}{q}=
{q^{o(1)}Z^{5/2}C^2}=q^{o(1)}Z^{5/2}\frac{XL}q.	
\end{split}
\end{equation}

\subsubsection{$\mathrm{Err}_2$} We refer to \eqref{eqErr3} for notations and the shape of this term; in particular $c\leq C/q$. This time the two functions
$$n\mapsto e(\frac{an}{cq^2}),\ m\mapsto K(mr^2)e(\frac{-a\ell m}{cq^2})$$
have modulus $cq^2\leq Cq$. Under Voronoi, for $(a,cq)=1$ the first function transforms (essentially) into a Kloosterman sum of the shape 
\begin{equation}\label{dualn3}\Kl_2(\pm \ov an;cq^2)	
\end{equation}
 while the second sum transforms into 
\begin{equation}\label{dualm3}\frac{1}{q^{1/2}}\sum_{b\mods q}\what K(b)e(\pm m\frac{\ov{a\ell}(1-\ov{a\ell}br^2cq)}{cq^2})=K(\pm (\ov{a\ell})^2 r^2m)e(\pm \frac{\ov {a\ell} m}{cq^2}).	
\end{equation}
The sum over $a\mods {cq^2},\ (a,cq)=1$ of the product of \eqref{dualn3}
and \eqref{dualm3} equals
$$ \sumstar_{a\mods{cq^2}}K(\pm (\ov{a\ell})^2 r^2m)e(\pm \frac{ \ov{a\ell} m}{cq^2})\Kl_2(\pm  \ov an;cq^2)$$
$$=\bigl(\sumstar_{a\mods c}e(\pm\frac{\ov{a\ell q^2}m}{c})\Kl_2(\pm  \ov{aq^4}n;c)\bigr)\bigl(\sumstar_{a\mods{q^2}}K(\pm \ov{a\ell}^2 r^2m)e(\pm \frac{ \ov{a\ell c} m}{q^2})\Kl_2(\pm  \ov {ac^2}n;q^2)\bigr).$$
The first factor is (essentially) 
$$c^{1/2}e(-\frac{\ov{q}^2 n\ell \ov{m}}{c})$$ while the second factor equals
$$\sumstar_{a\mods{q^2}}K(\pm{a}^2 r^2m)e(\pm \frac{ a\ov{ c} m}{q^2})\Kl_2(\pm  a\ov {c^2}\ell n;q^2)=$$
$$\sumstar_{a\leq q}K(\pm{a}^2 r^2m)e(\pm \frac{ a\ov{ c} m}{q^2})\sum_{b\mods q}e(\pm \frac{ b\ov{ c} m}{q})\frac1q\sumstar_{x\mods {q^2}}
e(\frac{\pm (a+bq)\ov {c^2}\ell nx+\ov x}{q^2})$$
$$=\sumstar_{a\leq q}K(\pm{a}^2 r^2m)e(\pm \frac{ a\ov{ c} m}{q^2})\sum_\stacksum{x\mods {q^2}}{x\equiv -cm\ov{\ell n}\mods q}
e(\frac{\pm a\ov {c^2}\ell nx+\ov x}{q^2})$$
$$=\sumstar_{a\leq q}K(\pm{a}^2 r^2m)e(\pm \frac{ a\ov{ c} m}{q^2})e(\mp \frac{ a\ov{ c} m}{q^2})e(-\frac{\ov{cm}\ell n}{q^2})\sum_{y\mods q}e(\frac{\pm a\ov {c^2}\ell ny-\ov{cm}^2(\ell n)^2y}{q})$$
$$=qK(\ov{m}^3\ell^2r^2n^2)e(\frac{-\ov{c}n\ell \ov{m}}{q^2}).$$
Hence the sum over $a\mods {cq^2},\ (a,cq)=1$ of the product of \eqref{dualn3}
and \eqref{dualm3} equals (essentially) to
$$c^{1/2}q\cdot K(\ov{m}^3\ell^2r^2n^2)e(\frac{-n\ell \ov{m}}{cq^2}),$$
and $\mathrm{Err}_2$ is (essentially) transformed into
\begin{equation*}
\begin{split}
& \frac{q^{o(1)}}{LCq}\sum_{c\leq C/q}\frac{1}{c^{1/2}}\sum_{\ell\in\rmL}\lambda(\ell,1)\frac{XL}{(cq^2Z)^{3/2}}\sum_{n\sim \frac{(cq^2Z)^3}{XL}}\lambda(n,1)\frac{X}{cq^2Z}\sum_{m\ll \frac{(cq^2Z)^2}{X}}\lambda_f(m) K(\ov{m}^3\ell^2r^2n^2)e(\frac{-n\ell \ov{m}}{cq^2})\\
\ll& \frac{q^{o(1)}X^2}{Z^{5/2}Cq^6}\bigg|\sum_{c\leq C/q}\frac{1}{c^{3}}\sum_{n\sim Z^3(XLq^3)^{1/2}}\lambda(n,1)\sum_{\ell\in\rmL}\lambda(\ell,1)\sum_{m\ll Z^2qL}\lambda_f(m) K(\ov{m}^3\ell^2r^2n^2)e(\frac{-n\ell \ov{m}}{cq^2})\bigg|.
\end{split}
\end{equation*}
We just need to save a bit more than $O(q^\delta L)$ from the trivial estimate $\mathrm{Err}_2\ll q^{o(1)}XL$.

As in the treatment of the generic term $\mathrm{Main}_0$, we can then apply the Cauchy--Schwarz inequality to smooth the $n$-sum, but now we can put the $c$-sum outside the square and the $(m,\ell)$-sums inside the square. After this, applying Poisson summation to the $n$-variable leads to the following
\begin{equation*}
\begin{split}
&\sum_{n\sim Z^3(XLq^3)^{1/2}}K(\ov{m}^3\ell^2r^2n^2)\ov{K(\ov{m'}^3\ell'^2r^2n^2)}e(\frac{-n\ell \ov{m}}{cq^2})e(\frac{n\ell' \ov{m'}}{cq^2})=O_A(q^{-A})+\\
&\frac{Z^3(XLq^3)^{1/2}}{cq^2}\sum_{\beta\mods{q^2}}K(\ov{m}^3\ell^2r^2\beta^2)\ov{K(\ov{m'}^3\ell'^2r^2\beta^2)}e(\frac{-(\ell \ov{m}-\ell' \ov{m'})\beta \ov{c}}{q^2})\sum_{\gamma \mods c}e(\frac{-(\ell \ov{m}-\ell' \ov{m'})\ov{q^2}\gamma }{c}).
\end{split}
\end{equation*}
Here we have noticed that in the dual sum only the $n=0$ zero-frequency contributes, since $$cq^2\leq q^{-\eta}(XLq^3)^{1/2}$$ for some $\eta>0$. Writing $\beta=\beta'+qy,\ \beta'\leq q,\ y\mods q$ we see that this is further equal to
\begin{gather*}
\frac{Z^3(XLq^3)^{1/2}}{q^2}\sum_{\beta\leq{q}}K(\ov{m}^3\ell^2r^2\beta^2)\ov{K(\ov{m'}^3\ell'^2r^2\beta^2)}e(\frac{-(\ell \ov{m}-\ell' \ov{m'})\beta \ov{c}}{q^2})\\
\ \ \ \ \times \sum_{y\mods{q}}e(\frac{-(\ell \ov{m}-\ell' \ov{m'})y \ov{c}}{q})\cdot \delta_{m\ell'\equiv m'\ell \mods{c}}\\
=\frac{Z^3(XLq^3)^{1/2}}{q}\sum_{\beta\leq{q}}K(\ov{m}^3\ell^2r^2\beta^2)\ov{K(\ov{m'}^3\ell'^2r^2\beta^2)}e(\frac{-(\ell \ov{m}-\ell' \ov{m'})\beta \ov{c}}{q^2})\cdot \delta_{m\ell'\equiv m'\ell \mods{cq}}.
\end{gather*}
Given $(m,m',\ell, \ell')$, we consider two cases. 

-- If $m\ell'=m'\ell$ , the $\beta$-sum above is $\ll q$. Such terms contribute to $\mathrm{Err}_2$ 
$$\frac{q^{o(1)}X^2}{Z^{5/2}Cq^6}\sum_{c\leq C/q}\frac{1}{c^{3}}(XLq^3)^{1/4}\left(\sum_{\ell\in\rmL}|\lambda(\ell,1)|^2\sum_{m\ll Z^2qL}|\lambda_f(m)|^2 \frac{Z^3(XLq^3)^{1/2}}{q}\cdot q\right)^{1/2}\ll q^{o(1)}X/q^{1/2}.$$

-- If $m\ell'\neq m'\ell$, we can write $m\ell'=m'\ell+cq\delta$ with $\delta\neq 0$, then the $\beta$-sum after Poisson summation is 
$$\sum_{\beta\mods{q}}K(m'{\ell}\ell'r^2\beta^2)\ov{K({m'}\ell^2r^2\beta^2)}e(\frac{\beta \delta}{q}).$$
We distinguish two further cases. 

-- If $\mcF$ is ``non-exceptional" in the sense of \cite[p. 1686]{FKM2} (i.e., $\mcF$ is not geometrically isomorphic to the product of a Kummer and an Artin--Schreier sheaf or in other terms $K$ is not proportional to the product of an additive and a multiplicative character $\mods q$) then for any $\alpha\in\Fqt$ the pull-back sheaf $[x\mapsto \alpha x^2]^*\mcF$ whose trace function is $x\mapsto K(\alpha x^2)$ is non-exceptional and by  \cite[Theorem 6.3]{FKM2} the above sum is $\ll q^{1/2}$ for every $\delta\mods q$ unless $\ell/\ell'\mods q$ belongs to a subset $B\subset\Fqt$ satisfying $|B|=O(1)$. Then such terms contribute to $\mathrm{Err}_2$ 
$$q^{o(1)}\frac{X^2}{Z^{5/2}Cq^6}\sum_{c\leq C/q}\frac{1}{c^{3}}(XLq^3)^{1/4}\left(\sumsum_\stacksum{\ell,\ell',m,m'}{m\ell'\equiv m'\ell \mods{cq}}\frac{Z^3(XLq^3)^{1/2}}{q}\cdot (q^{1/2}+q\delta_{\ell/\ell'\in B})\right)^{1/2}$$
$$\ll q^{o(1)}Z(XL)^{3/4}+q^{o(1)}ZX^{3/4}L^{1/4}q^{1/4}.$$

-- If $\mcF$ is exceptional then we may assume that $K(x)=\chi(x)e(\frac{kx}q)$ for $k\in\Fq$ and $\chi$ a (non-trivial) Dirichlet character and the $\beta$-sum equals
$$\chi(m'{\ell}\ell')\ov{\chi({m'}\ell^2)}\sumstar_{\beta\mods{q}}e(\frac{km'{\ell}(\ell'-\ell)r^2\beta^2+\beta \delta}{q})\ll q^{1/2}+q\delta_\stacksum{\ell\equiv \ell'\mods q}{\delta\equiv 0\mods q}.$$
In either cases we obtain
\begin{equation}\label{Err2bound}\mathrm{Err}_2\ll q^{o(1)} X/q^{1/2}+q^{o(1)}Z(XL)^{3/4}+q^{o(1)}ZX^{3/4}L^{1/4}q^{1/4}.	
\end{equation}

\subsubsection{$\mathrm{Err}_3$} Recall from \eqref{eqErr4} $\mathrm{Err}_3$ takes the following shape
\begin{equation*}
\begin{split}
      \mathrm{Err}_3=&\frac{1}{L^\star Cq}\sum_{\ell\in\rmL}\ov{\lambda(1,\ell)}\ell^{-iv}\sum_\stacksum{c\leq  2C}{\ell |c, (c,q)=1}\frac{1}{c}\sumstar_{u(cq)}\sum_{n=1}^{\infty}\lambda(r,n)e\left(n\frac{u}{cq}\right)n^{iv}U\left(\frac{n}{XL}\right)\\
        &\frac{1}{q^{1/2}}\sum_{b\mods q}\what K(b)\sum_{m=1}^{\infty}\lf(m)e\left(\frac{-(bcr^2/\ell+u)m}{cq/\ell}\right)m^{-iv}V\left(\frac{m}X\right)h\left(\frac{c}{C},\frac{n-m\ell}{C^2q}\right).
\end{split}\end{equation*}
We treat exactly the same manner as what we did for the main term $\mathrm{Main}_0$ in \eqref{S'sum-2}. The chief difference occurs when we apply Proposition \ref{vorGL2}  to the $m$-sum, as now the modulus of the additive character is $cq/\ell$ rather than $cq$. Instead of getting \eqref{eqafter1stvoronoi}, we obtain
\begin{equation*}
\begin{split}
              \mathrm{Err}_3=&\frac{X}{L^\star Cq^2}\sum_{\pm}\sum_{\ell\in\rmL}\ov{\lambda(1,\ell)}\ell^{-iv}\sum_\stacksum{c\leq  2C}{\ell |c, (c,q)=1}\frac{\ell}{c^2}\frac{1}{q^{1/2}}\sumsumstar_{b(q),u(cq)}\what K(b)\\
        &\times\sum_{m=1}^{\infty}\ov{\lf(m)}e(\frac{\pm\ov{bcr^2/\ell+u}m}{cq/\ell})\sum_{n=1}^{\infty}\lambda(r,n)e\left(n\frac{u}{cq}\right)n^{iv}U\left(\frac{n}{XL}\right)\widehat{\mathcal{V}}^{\pm}_n\left(\frac{mX}{(c/\ell)^2q^2}\right).
        \end{split}\end{equation*}
        The extra factor $\ell$ in the summand of the $c$-sum will compensate with the congruence condition $\ell |c$, as compared to $\mathrm{Main}_0$, but now the effective length of the dual $m$-sum is $m\asymp M/L^2$, reduced by a factor $L^2$ as compared to $m\asymp M $ for the case of $\mathrm{Main}_0$; cf. \eqref{MNdef}. 
        
        Then we proceed as what we did for $\mathrm{Main}_0$, and the sum $\mathcal{C}(m,n;\frac{rcq}{n_1})$ in \eqref{eqaftervoronoi} will be changed into $\mathcal{C}(m\ell^2,n;\frac{rcq}{n_1})$. Correspondingly, the parameter $m$ in \S \ref{computation-fourier-transform} will have to be changed into $m\ell^2$ (and $m'\rightarrow m'{\ell'}^2$). The change of the parameter $m\rightarrow m\ell^2$ in the character sum has the effect of swapping the parameters $\ell$ and $\ell'$ in the congruence conditions in \S \ref{contribution-of-zero}, and everything else essentially remains the same. Since the length of the $m$-sum is reduced from $M$
to $M/L^2$, by book-keeping eventually we will save some factor $O(L^{\eta}), \eta>0$ for $\mathrm{Err}_3$, compared to the bound that we obtained for $\mathrm{Main}_0$.

\subsection{Conclusion}Collecting the bounds \eqref{final-for-main}, \eqref{Err1bound} and \eqref{Err2bound} and \eqref{SVKsecond} into  \eqref{S'sumErr} we obtain that
\begin{multline*}
S_{V,r}(K,X)\ll {q^{o(1)}}\frac{X}{L}r^{\theta_3}L^{\theta_2}+q^{o(1)+\theta_3}rZ^2\frac{XL}q\\+q^{o(1)}r^{1/2}\left(Z\frac{X^{3/4}q^{3/4}}{L^{1/4}}+Z^2L^{1/2}X^{1/2}q^{5/4}
+ L^{1/2}Xq^{-1/4}+ZL^{3/4}X^{3/4}q^{1/2}\right)	
\end{multline*}
and on taking $$L=q^{1/4}$$ to equate the first and fourth terms inside the parentheses we get
$$S_{V,r}(K,X)\ll q^{o(1)}r^{1/2}\left(Z{X^{3/4}q^{11/16}}+Z^{2}X^{1/2}q^{11/8}
+Xq^{-1/8}+q^{\theta_3}r^{1/2}{Z^2}{X}{q^{-3/4}}\right).$$
Replacing $X$ by $X/r^2$ and averaging this bound over $r\leq R$ we obtain
$$\sum_{r\leq R}S_{V,r}(K,X/r^2)\ll q^{o(1)}\left(Z{X^{3/4}q^{11/16}}+Z^{2}R^{1/2}X^{1/2}q^{11/8}
+Xq^{-1/8}+{Z^2}{X}{q^{\theta_3-3/4}}\right).$$
Combining this with \eqref{large-R}, we have
\begin{multline*}
\sum_{r\leq R}S_{V,r}(K,X/r^2)+\sum_{r\geq R}S_{V,r}(K,X/r^2)\\
\ll q^{o(1)}\left(Z{X^{3/4}q^{11/16}}+Z^{2}R^{1/2}X^{1/2}q^{11/8}
+Xq^{-1/8}+{Z^2}{X}{q^{\theta_3-3/4}}+XR^{\theta_3-1}\right),	
\end{multline*}
which upon choosing 
\begin{equation}\label{choice-of-R}
R=\left(\frac{X}{Z^4q^{11/4}}\right)^{\frac{1}{3-2\theta_3}}
\end{equation}
to equate the second and the fifth terms
gives
$$S^t_V(K,X)\ll q^{o(1)}\left(Z{X^{3/4}q^{11/16}}+Z^{\frac{4(1-\theta_3)}{3-2\theta_3}}X^{\frac{2-\theta_3}{3-2\theta_3}}q^{\frac{11(1-\theta_3)}{4(3-2\theta_3)}}
+Xq^{-1/8}\right).$$
  
  This completes the proof of Theorem \ref{thmStVbound}.
  
  Plugging \eqref{choice-of-R} into our previous assumption $R<L=q^{1/4}$ in \eqref{Lbound} we see the restriction
  $$Z^4q^{11/4}<X<Z^4q^{\frac{7-\theta_3}{2}}.$$
  
 \section{Square-root cancellation for certain exponential sums} \label{sec-sqroot}

In this section we establish Proposition \ref{sqrootcancel}. Let $K$ be the trace function of an irreducible middle extension sheaf $\mcF$ on ${\mathbb P}_\Fq^1$ of conductor $C(\mcF)$. For $\alpha,\beta,\alpha',\beta'\in\Fqt$ we define
\begin{equation}\label{L-sum}
L(u;q):=\frac{1}{\sqrt{q}}\sum_{b(q)}\widehat{K}(b)e\left(\frac{\alpha\overline{(b+\beta u)}}{q}\right)=\frac{1}{\sqrt{q}}\sum_{a}K(a)\Kl_2(\alpha a;q)e(-\frac{\beta au}q),
\end{equation}
and define $L'(u;q)$ in the same way with $\alpha',\beta'$ instead. We also set
\begin{equation}\label{Zcompute}
Z(v)=\frac{1}{q^{1/2}}\sum_{x\in\Fqt}K(xv)\Kl_2(\alpha xv;q)\Kl_2(\beta x;q).	
\end{equation}
and define $Z'(v)$ likewise using $\alpha',\beta'$. In this section we provide bounds for the sums
\begin{equation}\label{Zcorrelation}
\sum_{v}Z(v)\ov{Z'(v-\delta)}	
\end{equation}
for $\delta\in\Fq$ with different methods depending on whether $\delta=0$ or not. 

We recall the key assumptions:
\begin{itemize}
\item (MO)  There is no $ \lambda \in \Fqt$ such that the geometric monodromy group of $\mcF$ has some quotient which is equal, as a representation of $\pi_1$ into an algebraic group, to the geometric monodromy representation of $[\times \lambda]^*\KL_2$ modulo $\pm 1$.
\item (SL) The local monodromy representation of $\mcF$ at $\infty$ has no summand with slope $1/2$.
\end{itemize}

\subsection{The case $\delta=0$}

\begin{lemma} Assume $\mcF$ satisfies (MO). 
 Then 
 \[ L(u;q) =O (c(\mcF)^{O(1)}).\] 
 \end{lemma}
 
 \begin{proof} By Deligne's theorem, because $\mcF$ and $\HYPK_2 ( \alpha a ) \otimes \mathcal L_\psi( \beta a)$ are irreducible, this cancellation holds unless $\mcF$ is geometrically isomorphic to $\HYPK_2 ( \alpha a ) \otimes \mathcal L_\psi( \beta a)$, which clearly violates assumption (MO). \end{proof} 
 
\begin{lemma} We have \[\left| \sum_u L(u ;q)   \right| \leq c(\mcF)q^{1/2}. \]\end{lemma}

\begin{proof} \[ \sum_u L(u ;q) =q^{-1/2} \sum_u \sum_a K(a)\Kl_2(\alpha a;q)e(-\frac{\beta a u}q) = q^{1/2} K(0) \Kl_2(0)= K(0)q^{1/2} \]
and \[ | K(0)| \leq c(\mcF).\] \end{proof} 
For the third Lemma we need to define the following group
 \begin{equation}\label{deftorusF}
 T_\mcF(\Fq)=\{\lambda\in\Fqt,\ [\times\lambda]^*\mcF\hbox{ is geometrically isomorphic to }\mcF\}.
 \end{equation}
\begin{lemma} Suppose $\mcF$ satisfies (MO). Then
\[ \sum_{u(q)} L(u ;q)\ov{L'(u ;q)}  = O ( c(\mcF)^{O(1)} q^{1/2}) \]
unless $$\alpha/\alpha'=\beta/\beta'\in T_\mcF(\Fq).$$

In that later case we have
\[ \sum_{u(q)} L(u ;q)\ov{L'(u ;q)}  =c_\mcF(\alpha/\alpha')q+ O ( c(\mcF)^{O(1)} q^{1/2}) \]
for $c_\mcF(\alpha/\alpha')$ a complex number of modulus $1$. In particular if $T_\mcF(\Fq)=\{1\}$ we have
\[ \sum_{u(q)} L(u ;q)\ov{L'(u ;q)}  =\delta_\stacksum{\alpha=\alpha'}{\beta=\beta'}q+ O ( c(\mcF)^{O(1)} q^{1/2}). \]
\end{lemma}

\begin{proof} By Plancherel formula the sum equals
\begin{equation}\label{Lcorsum}
\sum_{a\mods q}K(\beta^{-1}a)\Kl_2((\alpha/\beta) a;q)\ov{K(\beta'^{-1}a)\Kl_2((\alpha'/\beta') a;q)}	
\end{equation}

Square-root cancellation now follows unless there is a nontrival map from $[\times\beta^{-1}]^*\mcF \otimes \mathcal [\times\beta'^{-1}]^*\mcF^\vee$ to $[\times\alpha/\beta]^*\HYPK_2 \otimes [\times\alpha'/\beta']\HYPK_2 $. If a nontrivial map exists, then an irreducible component on the left side must be isomorphic to an irreducible component on the right.

If the irreducible component on the right is trivial, then there is a nontrivial map from  $[\times\beta^{-1}]^*\mcF \otimes \mathcal [\times\beta'^{-1}]^*\mcF^\vee$  to the constant sheaf $\Ql$, hence a nontrivial map from $[\times\beta^{-1}]^*\mcF$ to $[\times\beta'^{-1}]^*\mcF$, which must be an isomorphism because both sides are irreducible, giving $\beta/\beta'=\lambda$. Furthermore, because $\Ql$ appears on the other side and $\KL_2$ is not geometrically isomorphic to any non-trivial multiplicative twist of itself \cite{GKM}, we must have $\alpha/\beta = \alpha'/\beta'$ so $\alpha/\alpha'=\beta/\beta'=\lambda$. 

If the irreducible component on the right is nontrivial, then its monodromy must be a nontrivial quotient of the monodromy of the product of two $\HYPK_2$, which is either $SL_2$ or $SL_2 \times SL_2$ depending on whether the two copies are isomorphic or not (by the Goursat-Kolchin-Ribet criterion in the later case). In either case, it has a further quotient equal to $PGL_2$,  with $\pi_1$ acting by its geometric monodromy action on one of the Kloosterman sheaves, modulo $\pm 1$. This quotient must also appear, as a $\pi_1$-representation, as a quotient of the monodromy on the left side. Because it is a simple group, it must appears as a quotient of the monodromy of either  $[\times\beta^{-1}]^*\mcF $ or $ [\times\beta'^{-1}]^*\mcF$. This implies that assumption (MO) is violated. \end{proof}

\subsection{The case $\delta\not=0$}

\begin{proposition}\label{sqrootcanceldeltanot=0}\label{hardest-exponential-sum} Assume that $\mcF$ satisfies both (MO) and (SL). Then for any $\alpha,\beta,\alpha',\beta',\delta\in\Fqt$ we have
\begin{equation}\label{deltanot=0sqbound}
\sum_{v}Z(v)\ov{Z'(v-\delta)}=O(q^{1/2})	
\end{equation}
Here the implicit constants depend only on $C(\mcK)$.  
\end{proposition}

We will prove this in several steps, using a series of lemmas. Let us first observe that by \eqref{Zcompute}, $Z(v)$ is the trace function of the sheaf 
\begin{equation}\label{mcZdef}
\mcZ:=(\mcF \otimes [\times \alpha]^*\HYPK_2 ) \star [y\ra\beta/y]^*\HYPK_2	
\end{equation} where $\star$ is multiplicative convolution. Our lemmas will focus mostly on understanding the geometry of this sheaf convolution and to compare with the sheaf $\mcZ'$ defined using the parameters $\alpha',\beta'$.
 We denote
 $$\mcK=\mcF \otimes [\times \alpha]^*\HYPK_2,\ \mcL=[y\ra\beta/y]^*\HYPK_2$$
 and $$\mcK'=\mcF \otimes [\times \alpha']^*\HYPK_2,\ \mcL'=[y\ra\beta'/y]^*\HYPK_2.$$

\begin{lemma}\label{he-irreducibility} If $\mcF$ satisfies (MO), then $\mcZ$ is geometrically irreducible \end{lemma}

\begin{proof} It follows from Goursat's lemma that, if $\mcF$ satisfies (MO), the monodromy of $\mcF \otimes [\times \alpha]^*\HYPK_2 $ is the product of the monodromy group of $\mcF$ and the monodromy group of $[\times \alpha]^*\HYPK_2$. Because it is a tensor product of irreducible representations, it is an irreducible representation of the product group. Then because $[y\ra\beta/y]^*\HYPK_2$ is an object of dimension $1$ in the sheaf convolution Tannakian category, convolving with it preserves irreducibility.  (Alternately, we can view the convolution as a Fourier transform, change of variables $y \to y^{-1}$, then another Fourier transform, and each of these steps preserve irreducibility.) \end{proof}

It follows immediately  from Lemma \ref{he-irreducibility} that, if $\mcF$ satisfies (MO), then the bound of  \eqref{deltanot=0sqbound} holds unless $\mcZ'$ is geometrically isomorphic to $[+\delta]^* \mcZ $. So this case is what we will focus on eliminating.

\begin{lemma}\label{he-lisse} If $\mcF$ satisfies (SL), then $\mcZ$ is lisse away from $\{0, \beta/\alpha, \infty\}$ \end{lemma}

\begin{proof} The middle convolution $\mcK \star \mcL $ is equal to the compactly supported convolution $\mcK \star_{!} \mcL$ up to a lisse sheaf, so it suffices to prove this for the compactly supported convolution. The compactly supported convolution is 
$$ R \pi_! ( (\mcF  \otimes [\times \alpha]^*\HYPK_2) \otimes [(x,v)\ra\beta x/v]^*\HYPK_2)$$
 where $\pi: \Gm \times \Gm \to \Gm$ is the projection from the torus with coordinates $(x,v)$ to the torus with coordinate $v$. To prove lisseness, we apply Deligne's semicontinuity theorem and examine the variation with $v$ of the Euler characteristic of the sheaf (in the $x$-variable) $(\mcF  \otimes [\times \alpha]^*\HYPK_2) \otimes [x\ra\beta x/v]^*\HYPK_2$

Let $V$ be an indecomposable summand of the local monodromy representation of $\mcF$ at $\infty$. By assumption (SL), the slope of $V$ is not  $1/2$.

If the slope of $V$ is $>1/2$, then $V \otimes [\times \alpha]^*\HYPK_2$ has the same slope, and thus for any given $v$, the sheaf in the $x$-variable, $(V  \otimes [\times \alpha]^*\HYPK_2) \otimes [x\ra\beta x/v]^*\HYPK_2$ has the same slope, which in particular is independent of $v$, so the contribution of this representation to the Swan conductor is constant.

If the slope of $V$ is $<1/2$, then the slope of $V  [\times \alpha]^*\HYPK_2$  is exactly $1/2$, and the slope of $(V  \otimes [\times \alpha]^*\HYPK_2) \otimes [x\ra\beta x/v]^*\HYPK_2$  is at most $1/2$. The slope of $(V  \otimes [\times \alpha]^*\HYPK_2) \otimes [x\ra\beta x/v]^*\HYPK_2$  is less than $1/2$ if and only if $ [\times \alpha]^*\HYPK_2 \otimes [x\ra\beta x/v]^*\HYPK_2$ has a summand of slope $<1/2$. By known properties of the Kloosterman sheaf \cite[10.4.5]{GKM}, this happens if and only if $\alpha = \beta v^{-1}$.

Because the Swan conductor at $\infty$ of $(\mcF \otimes [\times \alpha]^*\HYPK_2) \otimes [x\ra\beta x/v]^*\HYPK_2$ is constant away from $v=\beta/\alpha$, and its Swan conductor elsewhere is constant, by Deligne's semicontinuity theorem its cohomology is lisse (in the $v$ variable) away from $v=\beta/\alpha$. \end{proof}

\begin{lemma}\label{he-nontrivial} If $\mcF$ satisfies (SL), then $\mcK \star \mcL $ has a nontrivial singularity at zero. \end{lemma} 

\begin{proof} Let us first check that, if $\mcF$ satisfies (SL), at least one of the following four conditions must be satisfied: 

\begin{enumerate}

\item $\mcF$ is singular at some point on $\Gm$.

\item The local monodromy representation of $\mcF$ at zero is not unipotent.

\item The local monodromy representation of $\mcF$ at zero has a Jordan block of size $\geq 3$.

\item The local monodromy representation of $\mcF$ at $\infty$ has a summand with slopes $>1/2$.  \end{enumerate}

This follows from an Euler characteristic calculation. Assume none of these happen. Because $\mcF$ is a nontrivial irreducible middle extension sheaf on $\mathbb A^1$, its Euler characteristic is nonpositive. Because $\mcF$ is unipotent at $0$, it is tamely ramified at $0$. Because $\mcF$ is lisse on $\Gm$ and tamely ramified at $0$, its Euler characteristic is its rank, minus its drop at $0$, minus its Swan conductor at $\infty$. Because all the local monodromy at $0$ is unipotent, with unipotent blocks of size at most $2$, its drop at $0$ is at most half the rank. Because its slopes at $\infty$ are $<1/2$, its Swan conductor at $\infty$ is less than half the rank. So the Euler characteristic is positive, giving a contradiction.

So at least one of (1) to (4) must happen. This imply corresponding conditions for $\mcK=\mcF \otimes [\times \alpha]^*\HYPK_2$. We must have either

\begin{enumerate}

\item $\mcK$ is singular at some point on $\Gm$.

\item The local monodromy representation of $\mcK$ at zero is not unipotent.

\item The local monodromy representation of $\mcK $ at zero has a Jordan block of size $\geq 4$.

\item The local monodromy representation of $\mcK$ at $\infty$ has a summand with slopes $>1/2$.  \end{enumerate}

These follow by straightforward arguments. The most subtle are that, in case (2), we must check that the tensor product of a non-unipotent representation with a nonzero unipotent representation is non-unipotent, and in case (3), we must check that a Jordan block of size $\geq 3$ tensored with a Jordan block of size $\geq 2$ produces a Jordan block of size $\geq 4$. Both follow from standard representation theory.

We will now show, in each of these cases, that $\mcK \star \mcL $ has nontrivial local monodromy at zero. In fact, we will break case (2) into the cases (2w) where the local monodromy at $0$ is wild and (2t) where the local monodromy at $0$ is tame but not unipotent. Rojas-Le\'on has shown \cite[Theorem 16]{arl} that the wild part of the local monodromy at $0$ of a sheaf convolution $\mcK \star \mcL$ is given by a sum of the values of certain functors applied to the local monodromy representations of $\mcK$ and $\mcL$ at different points. We will show, in cases (1), (2w), and (4), that one of those functors produces a nontrivial value on $\mcK $ and $ \mcL $ and thus the local monodromy representation of $\mcK \star \mcL $ at $0$ is wild. In cases (2t) and (3) we will show, using a different result of Rojas-Le\'on (encapsulating earlier work of Katz), that the local monodromy at 0 of $\mcK \star \mcL  $ contains a nontrivial tame component. In every case we will deduce it is nontrivial.

The functors appearing in \cite[Theorem 16]{arl} are defined by swapping $0$ and $\infty$ in the functors defined in \cite[Theorem 9]{arl}, which is harmless as $\Gm$ has a symmetry switching $0$ and $\infty$. Thus we will need to swap $0$ and $\infty$ when citing results from \cite{arl}.
\begin{enumerate}

\item Fix $s$ a singularity of $\mcK $ in $\Gm$. We must show that the function $\overline{\rho}_{(s,0)}$ applied to $(\mcK_{(s)}^{w},\mcL_{(0)}^w)$ is non-trivial. This follows from \cite[Proposition 12]{arl}, using the assumption that  $\mcK$ has a singularity at $s$, and the fact that $\mcL$ has non-trivial wild local monodromy at $0$.

\item[(2w)]   In this case we must show that the functor $\overline{\rho}_{(0,0)} $ applied to $(\mcK_{(0)}^{w}, \mcL_{(0)}^w )$ (the wild parts of the local monodromy representations of these two sheaves at $0$) is non-trivial. This follows from \cite[Proposition 11]{arl} which shows that this functor applied to any two nontrivial wild representations is nontrivial, our assumption that $\mcK$ has nontrivial wild local monodromy at $0$, and the fact that $\mcL$ has nontrivial wild local monodromy at $0$.

\item [(2t)] For each character $\chi$ of the tame fundamental group of $\Gm$, and for $U_k$ a unipotent Jordan block of size $k$, Rojas-Le\'on defines \cite[\S6]{arl} a polynomial $P_{\mcK,\chi}(T)$ associated to a middle extension sheaf $\mcK$, whose coefficient of $T^k$ is the multiplicity of $\mathcal L_\chi \otimes U_k$ as a direct summand of the local monodromy of $\mcK$ at $\infty$ if $k>0$ and the multiplicity of $\mathcal L_\chi \otimes U_{-k}$ as a direct summand of the local monodromy of $\mcK$ at $0$ if $k<0$, and such that $P_{\mcK,\chi}(1)$ is minus the Euler characteristic of $\mcK$. He proves that $P_{\mcK \star \mcL,\chi}(T) = P_{\mcK,\chi}(T) P_{\mcL,\chi}(T)$ \cite[Proposition 28]{arl}.

If the local monodromy representation of $\mcK $ at $0$ is tame but not unipotent, then it contains some $\mcL_\chi \otimes U_k$ for some $k>0$ and nontrivial character $\chi$. Thus $P_{ \mcK , \chi}(T)$ contains some term $T^{-k}$. Because $\chi$ is nontrivial, $\mathcal L_\chi$ does not appear in the local monodromy of $\mcL$ at $0$ or $\infty$, and so $P_{ \mcL, \chi}(T) = - \chi( \Gm,  \mcL)=1$. Thus $P_{ \mcK , \chi}(T) P_{ \mcL, \chi}(T)$ contains some term $T^{-k}$, so the local monodromy representation of $\mcK\star\mcL$ contains $\mathcal L_\chi \otimes U_k$ and is nontrivial .

\item[(3)]  Now taking $\chi$ to be trivial, we see that  $P_{ \mcK , 1}(T)$ contains some term of the form $T^{-k}$ for $k \geq 4$. On the other hand, because the local monodromy of $\mcL$ at $\infty$ is a unipotent block of size $2$,  $P_{ \mcL, 1}(T) = T^2$. Thus $P_{ \mcK , \chi}(T) P_{ \mcL, \chi}(T)$ contains some term of the form $T^{2-k}$ for $k \geq 4$, and thus the local monodromy representation of $\mcK\star\mcL$ contains $U_{k-2}$ and is nontrivial.

\item [(4)] In this case we must show that the functor $\overline{\rho}_{(\infty,0)} $ applied to $(\mcK_{(\infty)}^{w}, \mcL_{(0)}^w )$, the wild parts of the local monodromy representations of these two sheaves at $\infty$ and $0$ respectively, is nontrivial. This follow from \cite[Proposition 13]{arl} which implies (swapping $0$ and $\infty$) that $\overline{\rho}_{(\infty,0)}$ is nontrivial as soon as some slope $a$ of the first input is greater than some slope $b$ of the first input, our assumption that $\mcK$ has some slope $>1/2$ at $\infty$, and the fact that $\mcL$ has slope $1/2$ at $0$. 
\end{enumerate}

\end{proof}

 Let us adopt some notation. Let $U_k(\mcF)$ be the number of unipotent blocks of size $k$ in the local monodromy representations of $\mcF$ at $0$, let $\chi_{x}(\mcF)$ be the drop plus the Swan conductor at the point $x$, and let $N_s(\mcF)$ be the rank of the local monodromy representation of $\mcF$ at $\infty$ with slope $s$. 
 
\begin{lemma}\label{zero-monodromy} The dimension of the monodromy invariant subspace of  $\mcK\star\mcL$ at $0$ is   \[ \sum_{k \geq 2} U_k(\mcF) + \sum_{k \geq 4} U_{k} (\mcF) \] \end{lemma}

\begin{proof} The dimension of the monodromy invariants at $0$ is the number of unipotent blocks in the local monodromy representation at $0$, which is the sum of all the coefficients of negative powers of $T$ in 
$P_{ \mcK\star\mcL, 1}(T)$, which is the sum of all coefficients of powers $<-2$ of $T$ in $P_{ \mcK,1}(T)$, which is the number of unipotent blocks of size at least $3$ of $\mcK$. A unipotent block of $\mcF$ of size $k$ produces a unipotent block of $\mcF \otimes [\times \alpha]^*\HYPK_2 $ of size $k+1$ and one of size $k-1$, so it produces one unipotent block of size $3$ if $k\geq 2$ and another if $k \geq 4$.
\end{proof}

\begin{lemma}\label{nonzero-monodromy} Assume $\mcF$ satisfies (SL). The dimension of the monodromy invariant subspace of  $\mcK\star\mcL$ at $\beta/\alpha$ is  

\[4 \sw_0(\mcF) + 4 \sum_{x \in \Gm} \chi_x(\mcF) + \sum_s N_s(\mcF) (2s + 2\max(s, 1/2))  -\left(2 \sum_{k \geq 1} U_k (\mcF) + \sum_{k \geq 2} U_k(\mcF) + \sum_{k \geq 4} U_k(\mcF) \right)\] \end{lemma}

\begin{proof}  The dimension of the monodromy invariants at $\beta/\alpha$ is minus the Euler characteristic of \[ \mcK \otimes [ y \to \beta/(\alpha y) ]^* \mcL = \mcK \otimes [ \times \alpha]^* \HYPK_2   = \mcF \otimes [\times \alpha]^* \HYPK_2 \otimes[ \times \alpha]^* \HYPK_2 .\]  on $\mathbb G_m$, minus the dimensions of the spaces of local monodromy invariants at $0$ and $\infty$ of \[ \mcK \otimes [ y \to v/y ]^* \mcL =  \mcF \otimes [\times \alpha]^* \HYPK_2 \otimes [\times \beta/v]^* \HYPK_2  \] for \emph{generic} $v$. These extra terms come from taking the middle convolution and not the compactly-supported convolution.

 By combining the formula for the Euler characteristic of a middle extension sheaf with a calculation of the invariants and Swan conductor of a tensor product monodromy representation at each point, the Euler characteristic term is $4$ times the sum of the drop plus Swan at all the finite singularities of $\mcF$, plus the sum over all local monodromy representations at $\infty$ of rank $r$ and slope $s$ of $r (2s + 2 \max(s,1/2))$, plus $4$ times the Swan conductor of $\mcF$ at $0$.
 
 Because $ [\times \alpha]^* \HYPK_2 $ and $[\times \beta/v]^* \HYPK_2$ both have local monodromy representations at $0$ unipotent of rank $2$, the dimension of the local monodromy invariants of $ \mcF \otimes [\times \alpha]^* \HYPK_2 \otimes[ \times \alpha]^* \HYPK_2 $  at $0$ is $2$ for each unipotent block of size $1$ in the local monodromy of $\mcF$ at $0$, $3$ for each unipotent block of size $2$, and $4$ for each unipotent block of greater size.  
  
  Because the local monodromy of $  [\times \beta/v]^* \HYPK_2 $ at $\infty$ is irreducible and distinct for distinct $v$, for only finitely many $v$ can $\mcF \otimes [\times \alpha]^* \HYPK_2 \otimes [\times \beta/v]^* \HYPK_2$ have nonzero local monodromy invariants at $\infty$, meaning that for generic $v$ there are no local monodromy invariants.  \end{proof}

\begin{lemma}\label{invariants-equal} Assume $\mcF$ satisfies (SL). Suppose that the dimension of the monodromy invariants subspace of $\mcK\star\mcL$ at $0$ is equal to the dimension of the monodromy invariant subspace of  $\mcK'\star\mcL'$ at $\beta'/\alpha'$. Then $\mcF$ is lisse on $\Gm$, with the local monodromy at $0$ unipotent with all blocks of size $2,3,$ and $4$, all slopes of the local monodromy representation at $\infty$ at most  $1/2$, and $\chi(\mathbb A^1, \mcF)=0$.  \end{lemma}

  \begin{proof} By Lemmas \ref{zero-monodromy} and \ref{nonzero-monodromy} our assumption implies that  
  \[ \sum_{k \geq 2} U_k(\mcF) + \sum_{k \geq 4} U_{k} (\mcF) \] \[= 4 \sw_0(\mcF) + 4 \sum_{x \in \Gm} \chi_x(\mcF) + \sum_s N_s(\mcF) (2s + 2\max(s, 1/2))  -\left(2 \sum_{k \geq 1} U_k (\mcF) + \sum_{k \geq 2} U_k(\mcF) + \sum_{k \geq 4} U_k(\mcF) \right)\]
 
 This gives 
\begin{multline*}
  2 \sum_{k \geq 1} U_k (\mcF)  + 2\sum_{k \geq 2} U_k(\mcF) + \sum_{k \geq 3} U_k(\mcF)+ \sum_{k \geq 4} U_{k} (\mcF) \\= 4 \sw_0(\mcF) + 4 \sum_{x \in \Gm} \chi_x(\mcF) + \sum_s N_s(\mcF) (2s + 2\max(s, 1/2))
\end{multline*}

  On the other hand, we have
  
  \[ 0 \leq - \chi(\mcF) =  \sw_0(\mcF) +  \sum_{x \in \Gm} \chi_x(\mcF) + \sum_s N_s(\mcF) s - \sum_{k\geq 1} U_k (\mcF) \]
  
  and \[ \sum_{k \geq 1} k U_k(\mcF) \leq \rank(\mcF) = \sum_s N_s(\mcF) \leq \sum_s N_s 2 \max(s,1/2) \]
  
  Subtracting twice the first inequality plus the second inequality from our identity, we must have
  
  \[ - U_1 (\mcF)  - \sum_{k \geq 5} (k-4) U_k(\mcF) \geq 2 \sw_0(\mcF) + 2 \sum_{x \in \Gm} \chi_x(\mcF) \geq 0 .\]
  
  Thus, for the identity to hold, every inequality must be sharp. It follows that we must have $\chi(\mcF)=0$, $\sum_{k \geq 1} k U_k(\mcF) = \rank(\mcF) $ (meaning the local monodromy at $0$ is unipotent), $U_k(\mcF)=0$ for $k \neq 2,3,4$,  $\chi_x(\mcF)=0$ for $x \in \Gm$ (meaning $\mcF$ is lisse on $\Gm$, and $2 \max (s,1/2)=0$ for all $s$ with $s \leq 0$ (meaning all slopes at $\infty \leq 1/2$). \end{proof}

\begin{lemma}\label{unipotents-zero}  Under the assumptions of Lemma \ref{invariants-equal}, the local monodromy of $\mcK\star\mcL$ at $0$ is unipotent.  \end{lemma}

\begin{proof} We will calculate the wild part and the $\chi$-tensor unipotent part of the local monodromy at $0$ for each nontrivial character $\chi$ using \cite{arl}'s local convolution theory, and show that they are trivial, deducing that the local monodromy is unipotent. 

By Lemma \ref{invariants-equal}, $\mcF$ is lisse on $\Gm$ and tame at $0$, so $\mcF_{(s)}$ vanishes for all $s \in \Gm$ and $\mcF_{(0)}^{w} =0$. Similarly, $\mcL$ is lisse on $\Gm$ and tame at $\infty$, so $\mcL_{(s)}$ vanishes for all $s \in \Gm$ and $\mcL_{(\infty)}^{w}=0$. So the only term in \cite[Theorem 16]{arl} that could be nonvanishing is \[ \overline{\rho}_{\infty, 0} (\mcF_{(\infty)}^{w} , \mcL_{(0)}^{w}).\] But by Lemma \ref{invariants-equal}, $\mcF_{(\infty)}^{w}$ has all slopes $<1/2$, and $\mcL_{(0)}^{w}$ has slope $1/2$, so by \cite[Proposition 13]{arl} we have \[ \overline{\rho}_{\infty, 0} (\mcF_{(\infty)}^{w} , \mcL_{(0)}^{w})=0\] as well.

For $\chi$, because $\mcF$ is unipotent at $0$, $\mcF \otimes [\times \alpha]^*\HYPK_2 $ is unipotent at $0$, and because $\mcF$ has no slope exactly equal to $1/2$ at $\infty$,  $\mcF \otimes [\times \alpha]^*\HYPK_2 $ is totally wild at infinity, so $\mathcal L_\chi \otimes U_k$ does not appear in the local monodromy of  $\mcF \otimes [\times \alpha]^*\HYPK_2 $ at $0$ or $\infty$ for $\chi$ nontrivial. Thus $P_{ \mcK , \chi} ( T)$ is a constant, and by the same reasoning $P_{ \mcL, \chi} (T)$ is a constant, so their product is a constant, and thus by \cite[Proposition 28]{arl} $\mathcal L_\chi \otimes U_k$  does not appear in the local monodromy of $\mcK \star\mcL$ at zero or infinity.

\end{proof}

\begin{lemma}\label{not-unipotent} Under the assumptions of Lemma \ref{invariants-equal}, the local monodromy of $\mcK'\star\mcL'$ at $\beta'/\alpha'$ is not unipotent. \end{lemma}

\begin{proof} Assume for contradiction that this representation is unipotent.
We have
$$\mcK'=\mcK'\star ([y\ra\beta'/y]^*\HYPK_2\star [y\ra\beta'y]^*\HYPK_2)=\mcK'\star\mcL'\star [\times\beta']^*\HYPK_2$$

Using this identity, we apply \cite{arl}'s local convolution theory to calculate the slope $1/2$ part of the local monodromy of $\mcK'$ at $\infty$. Because $ [\times\beta']^*\HYPK_2$ is lisse away from $0$ and $\infty$ and tame at $0$, the only functors that contribute to the wild part of $\mcK'$ at $\infty$ are $\rho_{ (0,\infty)}, \rho_{ (\beta'/\alpha',\infty)},$ and $\rho_{(\infty, \infty)}$. Because the unique slope at $\infty$ of $ [\times\beta']^*\HYPK_2$ is $1/2$, by \cite[]{arl} the image of $\rho_{(0,\infty)}$ has slopes $<1/2$ and by \cite[Proposition 11]{arl}, the image of $\rho_{(\infty, \infty)}$ has slopes $<1/2$. So the slope $1/2$ part only arises from 
$$ \rho_{ (\beta'/\alpha',\infty)} \left ((\mcK'\star\mcL')_{(\beta'/\alpha')},  [\times\beta']^*{\HYPK_2}_{(\infty)}^{w} \right) .$$
 Because $\rho_{ (\beta'/\alpha',\infty)}$ is exact and $(\mcK'\star\mcL')_{(\beta'/\alpha')}$ is unipotent, this is an iterated extension of  
 $$ \rho_{ (\beta'/\alpha',\infty)} \left (1,  [\times\beta']^*{\HYPK_2}_{(\infty)}^{w} \right) = [\times\alpha']^*{\HYPK_2}_{(\infty)} .$$ (This identity can be checked by applying local convolution functors to calculate the local monodromy at $\infty$ of  $\delta_{\beta'/\alpha'} \star [\times\beta']^*{\HYPK_2} $, say. ) 

Then the slope $1/2$ part of the local monodromy of $(\mcF \otimes [\times\alpha']^*{\HYPK_2} )$ at $\infty$ is an iterated extension of $[\times\alpha']^*{\HYPK_2}$, hence $[\times\alpha']^*{\HYPK_2}$ tensor a unipotent representation.

Let $V$ be the local monodromy representation of $\mcF$ at $\infty$. By Lemma \ref{invariants-equal} and assumption (SL), $V$ has all slopes $<1/2$, so $V \otimes [\times\alpha']^*{\HYPK_2}$ has all slopes exactly $1/2$. Thus  $V \otimes [\times\alpha']^*{\HYPK_2}$ is $[\times\alpha']^*{\HYPK_2}$ tensored with a unipotent representation. Because $V$ is a summand of  
$$ V \otimes ( [\times\alpha']^*{\HYPK_2} \otimes [\times\alpha']^*{\HYPK_2}) = (V \otimes [\times\alpha']^*{\HYPK_2} ) \otimes [\times\alpha']^*{\HYPK_2},$$ it is a summand of 
$$ [\times\alpha']^*{\HYPK_2} \otimes [\times\alpha']^*{\HYPK_2} $$ tensored with a unipotent representation, and because it has slope $<1/2$, it is a summand of  the slope $<1/2$ part of $[\times\alpha']^*{\HYPK_2} \otimes [\times\alpha']^*{\HYPK_2}$ tensored with a unipotent representation. The slope $<1/2$ part of $[\times\alpha']^*{\HYPK_2} \otimes [\times\alpha']^*{\HYPK_2}$ is the sum of a trivial representation and quadratic character, so $V$ is a summand of the sum of a unipotent representation and another unipotent representation tensored with a quadratic character.

Hence $V$ is a sum of unipotent representations and unipotent representations tensored with the quadratic character. In particular, $\mcF$ is tamely ramified at $\infty$. So because it is irreducible and lisse on $\Gm$, tamely ramified at $\infty$, and unipotent at $0$, it must be the trivial sheaf, which contradicts our assumption that $\mcF$ is nontrivial. \end{proof}
 
\begin{proof}[Proof of Proposition \ref{hardest-exponential-sum}] 

It is clear from their constructions that the conductors of $\mcZ$ and $\mcZ'$ are $O ( c (\mcF)^{O(1)})$. 
By Lemma \ref{he-irreducibility}, $\mcZ'$  is irreducible. So we get the bound of Proposition \ref{hardest-exponential-sum} unless $\mcZ'$ is geometrically isomorphic to $[+\delta]^* \mcZ$. Suppose this is the case.

By Lemma \ref{he-nontrivial}, $[+\delta]^* \mcZ$ has a singularity at $-\delta$ which should be a singularity of $\mcZ'$. Since $\delta\not=0,\infty$, Lemma \ref{he-lisse} (applied to $\mcZ'$) imply that $-\delta=\beta'/\alpha'$ and the dimension of the monodromy invariant of $\mcZ$ and $\mcZ'$ respectively at $0$ and $\beta'/\alpha'$ are equal. By Lemma \ref{unipotents-zero}, the local monodromy of $[+\delta]^* \mcZ$ at $-\delta$ is unipotent while by Lemma \ref{not-unipotent}, the local monodromy of $\mcZ'$ at $\beta'/\alpha'$ is not unipotent, a contradiction.
\end{proof}

\section{Some special cases}\label{secdual}

In this section we discuss three cases of functions $K$ for which Theorem \ref{thmStVbound} does not apply directly either because $K$ is not a trace function or is a trace function associated to a ``bad" sheaf $\mcF$. 

\subsection{A general duality principle}
 We start with the general duality principle hinted in the introduction. Let $K:\Fqt\ra\Cc$ be some function; we define  the  Dirichlet series
 $$L(\vphi\times f\times K,s):=\sum_\stacksum{n,r\geq 1}{(nr^2,q)=1}\frac{\lambda(r,n)\lf(n)K(nr^2)}{(nr^2)^s},\ \Re s>1.$$
We will show that $L(\vphi\times f\times K,s)$ admits analytic continuation to $\Cc$  and satisfies a functional equation. We then use this functional equation to obtain non-trivial bounds for $S^t_V(\Kl_2,X)$ when the sheaf $\mcF$ is $\KL_2$.
 
 \subsubsection{The functional equation for standard $L$-functions}
 Let $\vphi$ and $f$ be $\GL_3$ and $\GL_2$ cusp forms of level one; to simplify the exposition and the shape of the functional equation we assume that $f$ is holomorphic of some weight $k\geq 2$. The Rankin--Selberg $L$-function $L(\vphi\times f,s)$ has analytic continuation to $\Cc$ and satisfies the functional equation
  $$\Lambda(\vphi\times f,s)=\eps(\vphi\times f)\Lambda(\ov\vphi\times f,1-s)$$
  where
  $$\eps(\vphi\times f)=i^{3k}=i^k=\pm 1$$
  is the root number;
 $$\Lambda(\vphi\times f,s)=L_\infty(\vphi\times f,s)L(\vphi\times f,s)$$
  is the $L$-function completed by the Archimedean local factor
 $$L_\infty(\vphi\times f,s)=\prod_\stacksum{i=1,2}{j=1,2,3}\Gamma_\Rr(s-\mu_{f,i}-\mu_{\vphi,j}),\ \Gamma_\Rr(s)=\pi^{-s/2}\Gamma(s/2)$$
 with
 $$\mu_{f,1}=-\frac{k-1}2,\ \mu_{f,2}=-\frac{k}2$$
 and $(\mu_{\vphi,j})_{j=1,2,3}$ being the Langlands parameters (of the principal series representation attached to $\vphi$) which satisfy
 $$\mu_{\vphi,1}+\mu_{\vphi,2}+\mu_{\vphi,3}=0$$
 (and belong to $(i\Rr)^3$ under the Ramanujan--Petersson conjecture); $\ov\vphi$ denotes the automorphic form dual to $\vphi$ (attached to the contragredient automorphic representation of $\vphi$) with parameters $(\ov\mu_{\vphi,j})_{j=1,2,3}$ and Hecke eigenvalues $(\ov{\lambda(r,n)})_{(r,n)}$. Under our assumption, $L_\infty(\vphi\times f,s)$ and $L_\infty(\ov\vphi\times f,s)$ have no poles for $\Re s>-1/2+\theta_3=-\frac{1}7$.
 
 For $\chi\mods q$ a non-trivial Dirichlet character, the twisted $L$-function
 $L(\vphi\times f\times\chi,s)$ has analytic continuation to $\Cc$ and satisfies the functional equation
 \begin{equation}\label{fcteqnchi}
 \Lambda(\vphi\times f\times\chi,s)=\eps(\vphi\times f\times\chi)\Lambda(\ov\vphi\times f\times\ov\chi,1-s)	
 \end{equation}
 where
 $$\Lambda(\vphi\times f\times\chi,s)=q^{3s}L_\infty(\vphi\times f,s)L(\vphi\times f\times\chi,s),$$
 and 
 $$\eps(\vphi\times f\times\chi)=\eps_\chi^6\eps(\vphi\times f)$$
with
 $$\eps_\chi=q^{-1/2}\sum_{x\in\Fqt}\chi(x)e(\frac{x}q)$$
being the normalized Gauss sum (notice that the Archimedean local factor does not depend on (the parity of) $\chi$ because $f$ is holomorphic; this is the main reason why we have made the simplifying assumption).

 \subsection{Functional equation for algebraically twisted $L$-functions}

  Let $K:\Fqt\ra\Cc$ be a function on $\Fqt$ (extended by $0$ to $\Fq$); we define the Mellin transform for $\chi\mods q$
  $$ \widetilde K(\chi)=\frac{1}{(q-1)^{1/2}}\sum_{x\in\Fqt}K(x)\ov\chi(x)$$
  and we define for $(n,q)=1$ the ``$\rm GL_6$-transform" of $K$ as
\begin{equation}\label{hatK6}
	\widecheck{K}^6(n)=\frac{1}{q^{1/2}}\sum_{x\in\Fqt}\Kl_6(nx;q)K(x)
\end{equation}
 where $\Kl_6$ is the hyper-Kloosterman sum in six variables
 $$\Kl_6(x;q)=\frac{1}{q^{5/2}}\sumsum_\stacksum{x_1,\cdots,x_6\in\Fqt}{x_1.\cdots.x_6=x}e(\frac{x_1+\cdots +x_6}{q}).$$
 
 \begin{proposition}\label{functional-algebraic}
  The completed series
 $$\Lambda(\vphi\times f\times K,s)=q^{3s}L_\infty(\vphi\times f,s)L(\vphi\times f\times K,s)$$
  has analytic continuation to $\Cc$ and satisfies the functional equation
 	\begin{eqnarray*}
	\Lambda(\vphi\times f\times K,s)&=&\eps(\vphi\times f)\Lambda(\ov\vphi\times f\times \widecheck{K}^6,1-s)+\\
	&+&\frac{q^{3s}}{(q-1)^{1/2}}\widetilde K(\chi_0)\Lambda^{(q)}(\vphi\times f,s)\\
	&-&\frac{\eps(\vphi\times f)q^{-3s}}{(q-1)^{1/2}}\widetilde K(\chi_0)\Lambda^{(q)}(\ov\vphi\times f,1-s)\\
\end{eqnarray*}
where $\chi_0=1_{|\Fqt}$ denotes the trivial character modulo $q$ and $\Lambda^{(q)}(\cdots)$ denotes the complete $L$-function with the local factor at $q$ being removed.
 \end{proposition}
\proof By the Mellin inversion formula
 $$K(nr^2)=\frac{1}{(q-1)^{1/2}}\sum_{\chi\mods q}\widetilde K(\chi)\chi(nr^2),$$ 
 we have
 $$\Lambda(\vphi\times f\times K,s)=\frac{q^{3s}}{(q-1)^{1/2}}\widetilde K(\chi_0)\Lambda^{(q)}(\vphi\times f,s)+\frac{1}{(q-1)^{1/2}}\sum_\stacksum{\chi\mods q}{\chi\not=\chi_0}\widetilde K(\chi)\Lambda(\vphi\times f\times\chi,s)$$
 and the result follows by applying the functional equations \eqref{fcteqnchi}, the identity
  $$\widecheck{K}^6(nr^2)=\frac{1}{(q-1)^{1/2}}\sum_{\chi\mods q}\eps_\chi^6\widetilde K(\chi)\ov\chi(nr^2),$$
   and the value of the Ramanujan sum
   $\eps_{\chi_0}=-q^{-1/2}.$
\qed

By standard contour shifts we deduce the following.
\begin{corollary}\label{smoothfcteqn} Let $V$ be a smooth compactly supported function satisfying \eqref{thmeqtestfct}; we have for any $X\geq 1$
	\begin{eqnarray*}
	\sum_{(nr,q)=1}\lambda(r,n)\lf(n)K(nr^2)V(\frac{nr^2}{X})&=&
	\eps(\vphi\times f)\frac{X}{q^3}\sum_{(nr,q)=1}\ov{\lambda(r,n)}\lf(n)\widecheck{K}^6(nr^2)\widecheck V^6(\frac{nr^2}{q^6/X})	\\
	&+&O(Z^B\frac{|\widetilde{K}(\chi_0)|}{q^{1/2-\theta_3}})
	\end{eqnarray*}
	where
	$$\widecheck V^6(x)=\intc_{(3/2)}\widetilde V(1-s)\frac{L_\infty(\ov\vphi\times f,s)}{L_\infty(\vphi\times f,1-s)}x^{-s}ds$$
	and
	$\widetilde V(s)=\int_0^\infty V(y)y^s\frac{dy}y$
	is the Mellin transform of $V$. Here $\theta_3=5/14$ is the best known bound towards the Ramanujan--Petersson conjecture on $\rm GL_3$, $B\geq 0$ is an absolute constant (any $B>5/2$ will do) and the implicit constant in the $O(\cdots)$ term depends on $\vphi,f$ and the implicit constants in \eqref{thmeqtestfct}.
\end{corollary}
\proof We have 
$$V(x)=\intc \widetilde V(s)x^{-s}ds$$
(the integration is along the vertical line $\Re s=1+1/14$) so that
$$\sum_{(nr,q)=1}\lambda(r,n)\lf(n)K(nr^2)V(\frac{nr^2}{X})=\intc \frac{\Lambda(\vphi\times f\times K,s)}{L_\infty(\vphi\times f,s)}\widetilde V(s)(\frac{X}{q^3})^sds$$
$$=q^3\eps(\vphi\times f)\intc  L(\ov\vphi\times f\times \widecheck K^6,1-s)\frac{L_\infty(\ov\vphi\times f,1-s)}{L_\infty(\vphi\times f,s)}\widetilde V(s)(\frac{X}{q^6})^sds$$
$$+\frac{\widetilde K(\chi_0)}{(q-1)^{1/2}}\intc L^{(q)}(\vphi\times f,s)\widetilde V(s)X^sds$$
$$-\frac{\eps(\vphi\times f)\widetilde K(\chi_0)}{(q-1)^{1/2}}\intc L^{(q)}(\ov\vphi\times f,1-s)\frac{L_\infty(\ov\vphi\times f,1-s)}{L_\infty(\vphi\times f,s)}
\widetilde V(s)(\frac{X}{q^6})^sds$$
upon applying Proposition \ref{functional-algebraic}.
In the first integral we make the change of variable $s\leftrightarrow 1-s$ getting
$$\eps(\vphi\times f)\frac{X}{q^3}\intc_{(-1/14)}  L(\ov\vphi\times f\times \widecheck K^6,s)\frac{L_\infty(\ov\vphi\times f,s)}{L_\infty(\vphi\times f,1-s)}\widetilde V(1-s)(\frac{X}{q^6})^{-s}ds$$
and shifting the contour back to $\Re s=3/2$ without hitting any poles we obtain the first sum. For the second and third integrals we shift the contour to $\Re s=0$ and apply trivial bounds; the term $q^{\theta_3}$ arise from the following bound for the inverse of the local factor
$$L_{q}(\vphi\times f,s)^{-1}\ll q^{\theta_3},\ \Re s=0.$$
\qed

In particular we deduce that
\begin{equation}\label{trivialgeneral}
\sum_{n,r}\lambda(r,n)\lf(n)K(nr^2)V(\frac{nr^2}{X})\ll_{\vphi,f}Z^Bq^{3+o(1)}\bigl(\|\widecheck{K}^6\|_\infty +\|K\|_\infty q^{\theta_3-3}\bigr).	
\end{equation}
for some absolute constant $B\geq 0$  and the implicit constant  depends on $\vphi,f$ and the implicit constants in \eqref{thmeqtestfct}. This bound is stronger than the trivial bound $O(X^{1+o(1)})$ as long as $$X\geq q^{3+\delta},\ \delta>0$$ and one can use it to prove the ``convexity bound" \eqref{ConvexKtwisted} mentioned in the introduction.

Applying Theorem \ref{thmStVbound} to the sum
$$\sum_{n,r}\ov{\lambda(r,n)}\lf(n)\widecheck{K}^6(nr^2)\widecheck V^6(\frac{nr^2}{q^6/X})$$
(after a dyadic partition of unity), one deduces that
\begin{corollary}\label{cordual} Notations being as in the above corollary; suppose that the transform $\widecheck{K}^6$ defined in \eqref{hatK6} is $L^1$-close to the trace function $K'$ of a good sheaf $\mcF'$, then we have 

\begin{multline*}
S^t_V(X)\ll  Z^{B'}\qoo\frac{X}{q^3}\bigg((q^6/X)^{3/4}q^{11/16}+(q^6/X)^{\frac{2-\theta_3}{3-2\theta_3}}q^{\frac{11(1-\theta_3)}{4(3-2\theta_3)}}
+(q^6/X)q^{-1/8})\\
+\|\widecheck{K}^6-K'\|_1(\frac{q^5}X+1)(\frac{q^6}X)^{\theta_3}\bigg).
\end{multline*}
Here  $\theta_3=5/14$
is the best known bound towards the Ramanujan--Petersson conjecture on $\rm GL_3$; $B'\geq 0$ is an absolute constant, $$\|\widecheck{K}^6-K'\|_1:=\sum_{x\in \Fqt}|\widecheck{K}^6(x)-K'(x)|,$$  and the implicit constant depends on $\vphi,f,|\widetilde K(\chi_0)|,\ C(\mcF')$ and the implicit constants in \eqref{thmeqtestfct}.
\end{corollary}

\begin{remark}
Ignoring the $Z$ parameter, the above bound is non-trivial as long as $$X\gg q^{3-1/12+\eta},\ \eta>0$$ and for $X=q^3$ we obtain
$$S^t_V(q^3)\ll q^{3-1/16+o(1)}.$$
\end{remark}

The typical situation to apply Corollary \ref{cordual} is when $K$ is the trace function attached to a geometrically  non-trivial and irreducible sheaf $\mcF$ of weight $0$. Unless $\mcF$ is geometrically isomorphic to a sheaf of the shape
$[x\mapsto \alpha x^{-1}]^*\KL_6$ for some $\alpha\not=0$, the function $\widecheck{K}^6$ is close to the trace function, $K'$ say, of a geometrically irreducible sheaf of weight $0$, namely the convolution sheaf $$\widecheck{\mcF}^6:=\KL_6\star [x\mapsto x^{-1}]^*\mcF$$ (alternatively $\widecheck{\mcF}^6$ is the sheaf obtained from $\mcF$ by applying $6$ times the composition of $[x\mapsto x^{-1}]^*$ and the geometric Fourier transform). In such a case one has
$$|\widetilde K(\chi_0)|,\ \|\widecheck{K}^6-K'\|_1=O_{C(\mcF)}(1).$$
We discuss three explicit examples below.

\subsection{Arithmetic progressions} 

One of the motivations for  investigating algebraic twists of $L$-functions is the level of distribution of their coefficients in arithmetic progressions of large modulus; specifically let 
$$\lambda_F(m)=\sum_{nr^2=m}\lambda(r,n)\lf(n);$$
one would like to prove  that for $(a,q)=1$ and  $V$ satisfying \eqref{thmeqtestfct}, one has
\begin{equation}\label{distriblevel}\sum_{m\equiv a\mods q}\lambda_F(m)V(\frac{m}X)=o_{\vphi,f,V}(\frac{X}q)
 \end{equation}
 for $q=X^\vartheta$ and some $\vartheta>0$ (a level of distribution) as large as possible (since $\vphi$ and $f$ are cuspidal there should be no main term).

Taking $$K(n)=q^{1/2}\delta_{n\equiv a\mods q}$$ we have for $(m,q)=1$
$$\widecheck K^6(m)=\Kl_6(am;q);$$
therefore applying Corollary \ref{smoothfcteqn} and bounding the resulting sums trivially (using Deligne's bound $|\Kl_6(am;q)|\leq 6$) we obtain
$$\sum_{nr^2\equiv a\mods q}\lambda(r,n)\lambda_f(n)V(\frac{nr^2}X)\ll_{V,\vphi,f} q^{5/2+o(1)}$$
which gives \eqref{distriblevel} as long as $$\vartheta<\vartheta_6=2/7.$$

To go beyond $\vartheta_6=2/7$ we would need, at least, to be able to improve the trivial bound for the sum $$\sum_{n,r}\ov{\lambda(r,n)}\lf(n)\Kl_6(anr^2;q)\widecheck V^6(\frac{nr^2}{X'})$$
 for
$X'=q^6/X$ a bit less than $q^{5/2}$. While $\Kl_6(an;q)$ is definitely ``good" (its geometric monodromy group is $\SL_6$ and its slopes at $\infty$ are equal to $1/6$), Theorem \ref{thmStVbound} is non-trivial only for
$$X'\gg q^{5/2+1/4+\delta},\ \delta>0.$$
This demonstrates the importance and difficulty of detecting cancellation for algebraically twisted sums of shorter length.
\begin{remark}
This level of distribution $\vartheta_6=2/7$ corresponds to Selberg's level of distribution $\vartheta_2=2/3$ for the divisor function (or the Fourier coefficients of modular forms) or $\vartheta_3=1/2$ for the ternary divisor function (see \cite{FrIwAnn,HBActa,FKMMath} for improvements in this last case).	
\end{remark}

\begin{remark}
 An Archimedean analog of this question is to improve the {\em Landau's type} bound  for the sharp-cut sum 
 $$\sum_{m\leq X}\lambda_F(m)\ll X^{5/7+o(1)}$$ (cf. \cite[Prop. 1.1]{FrIw} for the degree $m=6$ in the notations of that paper); under the Ramanujan--Petersson conjecture, this amounts to  non-trivial estimates for $$\sum_{m\leq X'}\ov{\lambda_F(m)}\,e(t(m/X')^{1/6})$$ when $X'$ is a bit less than $t^{5/2}$. This has recently been worked out by the first named author and Q. Sun in \cite{Lin-Sun} and this might suggest that if the modulus $q$ is a suitably factorable integer (either a smooth number or a large power of a fixed prime number) other exponential sum methods might allow to pass the $2/7$ barrier for such a composite modulus.

%
%
%
%
\end{remark}

\subsection{Additive characters}\label{secfourier}

We now discuss Theorem \ref{thmStVbound} when $K$ is a trace function of a sheaf $\mcF$ which is not Fourier: this implies that for $q$ large enough (depending on $C(\mcF)$) $K$ is proportional to an additive character $\psi(\bullet)=e(\frac{a\bullet}{q})$. 

In this case,  a much more general bound is expected  (cf. \cite{Miller}): for any $\alpha\in\Rr$, one should have
\begin{equation}\label{nontrivialadditive}
\sum_{n,r}\lambda(r,n)\lambda_f(n)e(\alpha nr^2)V(\frac{nr^2}{X})\ll_{\vphi,f,V} X^{1-\eta+o(1)}	
\end{equation}
for some $\eta>0$ (possibly $\eta=1/2$) and this should hold uniformly for $\alpha\in \Rr$. Such a bound would be analogous to Wilton's bound for Fourier coefficients of $\GL_2$ automorphic forms or Miller's bound for $\GL_3$ automorphic forms; however for ranks greater than $3$ such bounds are still unknown.  

As was pointed out in \cite{Miller} if $\alpha$ is a rational number the analytic properties of $L(\vphi\times f\times \chi,s)$ for $\chi$ a Dirichlet character of modulus dividing the denominator of $\alpha$ yield \eqref{nontrivialadditive} at least if $X$ is large compared to $q$. 
For instance, for $q$ a prime number, $(an,q)=1$ we have
$$\widecheck{e(\frac{a\bullet}q)}^6(n)=\Kl_5(-\ov an;q)-{(-1)^5}{q^{-3}}$$
and Corollary \ref{smoothfcteqn} together with Deligne's bound $|\Kl_5(an;q)|\leq 5$ yield \eqref{nontrivialadditive}  as long as
$$X\geq q^{3+\delta},\ \delta>0.$$
Now since the sheaf $[x\mapsto -\ov a x]^*\KL_5$ is good (it is Fourier, its geometric monodromy group is $SL_5$ and all its $\infty$-slopes equal $1/5$) we have 
$$\sum_{r,n}\lambda(r,n)\lf(n)e(\frac{anr^2}q)V(\frac{nr^2}{q^3})\ll q^{2+\theta_3+o(1)} +q^{3-1/16+o(1)}$$
(the first term on the right is the contribution of the $n$'s divisible by $q$) so that \eqref{nontrivialadditive} holds in the slightly wider range
$$X\geq q^{3-1/16+\delta},\ \delta>0.$$

\subsection{Kloosterman sums}\label{dualKl2} Finally we observe that the Kloosterman sheaf $\KL_2$ while being Fourier is definitely not good: neither (SL) nor (MO) is satisfied. On the other hand, 
$$\widecheck{\Kl_2}^6(n)=\Kl_4(n;q)-q^{-5/2}-q^{-7/2}$$
and the Kloosterman sheaf $\KL_4$
is good: its geometric monodromy group is $\SL_4$ and all its $\infty$-slopes are $1/4$. By  Corollary \ref{smoothfcteqn} we have
$$\sum_{r,n}\lambda(r,n)\lf(n)\Kl_2(nr^2;q)V(\frac{nr^2}{q^3})\ll_{\vphi,f,V}q^{3-1/16+o(1)}$$
which again is non-trivial as long as 
$$X\geq q^{3-1/16+\delta},\ \delta>0.$$

\begin{remark} Further examples of geometrically irreducible sheaves satisfying neither (MO) nor (SL) are the sheaves of the shape
$$\bigotimes_{i=1}^r\Sym_{k_i}\circ[\times\lambda_i]^*\KL_2$$
where $\Sym_{k}$ denotes the $k$-th symmetric power representation of $\SL_2$ and $\lambda_i,\ i=1,\cdots, r$ are distinct elements of $\Fqt$. The underlying trace function is
$$K:x\in\Fqt\mapsto \prod_{i=1}^r\sym_{k_i}(\theta_{\lambda_i x})$$
where $2\cos(\theta_{\lambda_i x})=\Kl_2(\lambda_i x)$ and $\sym(k\theta)=\sin((k+1)\theta)/\sin(\theta)$. In that case all the non-trivial slopes of the monodromy at $\infty$ are $1/2$ and the set of $\lambda$'s for which the condition in (MO) fails are precisely the $\lambda_i$. It should be possible to check for most or all of these sheaves $\mcF$ that $\widecheck{\mcF}^6$ is good.
\end{remark}

\end{document}